\documentclass[10pt]{amsart}
\usepackage{amssymb}

\theoremstyle{plain}
\newtheorem{prop}{Proposition}[section]
\newtheorem{thm}[prop]{Theorem}

\newtheorem{lemma}[prop]{Lemma}

\theoremstyle{definition}
\newtheorem{defi}[prop]{Definition}

\newtheorem{rem}[prop]{Remark}
\newtheorem{question}[prop]{Question}

\theoremstyle{remark}

\numberwithin{equation}{section}

\newcommand{\N}{\mathbb N}
\newcommand{\Z}{\mathbb Z}

\newcommand{\R}{\mathbb R}
\newcommand{\C}{\mathbb C}

\newcommand{\G}{\Gamma}
\newcommand{\ld}{\lambda}

\newcommand{\ba}{\backslash}

\newcommand{\Ot}{\operatorname{O}}
\newcommand{\SO}{\operatorname{SO}}

\newcommand{\so}{\mathfrak{so}}

\newcommand{\diag}{\operatorname{diag}}

\newcommand{\norma}[1]{\|{#1}\|_1}

\newcommand{\ww}{\theta}
\newcommand{\zz}{\ell}
\newcommand{\bs}{\mathbf{s}}

\title[Spectra of lens spaces]{Spectra of lens spaces from 1-norm spectra of congruence lattices}
\author{E. A. Lauret, R. J. Miatello and J. P. Rossetti}
\address{CIEM--FaMAF \\ Universidad Nacional de C\'ordoba\\ 5000-C\'ordoba, Argentina.}
\email{elauret@famaf.unc.edu.ar}
\email{miatello@famaf.unc.edu.ar}
\email{rossetti@famaf.unc.edu.ar}
\date{April 2015}

\begin{document}

\begin{abstract}
To every  $n$-dimensional lens space $L$, we associate a congruence lattice $\mathcal L$ in $\Z^m$, with $n=2m-1$ and
we prove a formula relating the multiplicities of Hodge-Laplace eigenvalues on $L$ with the number of lattice elements of a given $\norma{\cdot}$-length in $\mathcal L$.
As a consequence, we show that two lens spaces are isospectral on functions (resp.\ isospectral on $p$-forms for every $p$) if and only if the associated congruence lattices are $\norma{\cdot}$-isospectral (resp.\  $\norma{\cdot}$-isospectral plus  a geometric condition).
Using this fact, we   give, for every dimension $n\ge 5$, infinitely many examples of Riemannian manifolds that are isospectral on every level $p$ and are not strongly isospectral.
\end{abstract}

\maketitle

\section{Introduction}

Two compact Riemannian manifolds $M$ and $M'$ are said to be $p$-isospectral if the spectra of their Hodge-Laplace operator $\Delta_p$, acting on $p$-forms, are the same.
Many examples  of non-isometric isospectral manifolds have been constructed showing connections between the spectra and the geometry of a Riemannian  manifold (e.g.\ \cite{Mil}, \cite{Vigneras}, \cite{Ik80}, \cite{Go01}, \cite{Schueth}). In \cite{Su} Sunada  gave a general method that allowed constructing many examples; however, the resulting manifolds are always strongly isospectral, that is, they are isospectral for every natural, strongly elliptic operator acting on sections of a natural vector bundle over $M$; in particular, they are $p$-isospectral for all $p$  (see \cite{Ik83}, \cite{Gi}, \cite{Wo2} for applications in the case of spherical space forms).
The converse question is a problem that has been present for some time, i.e.\ whether $p$-isospectrality for all $p$ implies strong isospectrality (see for instance J.~A.~Wolf~\cite[p.~323]{Wo2}).
Manifolds that are $p$-isospectral for some values of $p$ only  have been investigated by several authors.
C.~Gordon \cite{Go} gave the first  example of this type in the context of nilmanifolds  and  A.~Ikeda \cite{Ik88} showed, for each $p_0>0$, lens spaces that are $p$-isospectral for all $p<p_0$ and are not $p_0$-isospectral. For more examples see \cite{Gornet}, \cite{MRp}, \cite{MRl}, \cite{GM}.

In this paper, we will show a connection between the spectra of lens spaces and the one-norm spectra of their associated congruence lattices, and, as a consequence,  we will give many pairs of  $p$-isospectral lens spaces for every $p$ that are far form being strongly isospectral.
These examples are the first of this kind to our best knowledge.

The main approach is as follows.
To each $2m-1$-dimensional lens space $L=L(q;s_1, \dots, s_m)$,  we associate the congruence lattice in $\Z^m$ defined by
\begin{equation*}
{\mathcal L=\mathcal L(q;s_1,\dots,s_m) = \{(a_1,\dots,a_m) \in \Z^m: a_1s_1+\dots +a_ms_m\equiv0\pmod q\}.}
\end{equation*}
In the first main result, Theorem~\ref{thm3:dim V_k,p^Gamma}, we give a formula for the multiplicity of each eigenvalue of $\Delta_p$ on a lens space, in terms of the multiplicities of the weights of certain representations of $\SO(2m)$ and the number of elements with a given $\norma{\cdot}$-length in the associated congruence lattice (see \eqref{eq3:dim V_k,p^Gamma}).
In particular, the multiplicity of the eigenvalue $k(k+n-1)$ of the  Laplace-Beltrami operator $\Delta_0$ simplifies to \eqref{eq3:dim V_k,1^Gamma}
\begin{equation*}
\sum_{r=0}^{\lfloor k/2\rfloor}\binom{r+m-2}{m-2} N_{\mathcal L}(k-2r),
\end{equation*}
where $N_{\mathcal L}(h)$ denotes  the number  of elements in $\mathcal L$ with $\norma{\cdot}$-length $h$.

As a consequence, we prove that two lens spaces $L$ and $L'$ are $0$-isospectral if and only if the associated lattices $\mathcal L$ and $\mathcal L'$ are  isospectral with respect to $\norma{\cdot}$ (Theorem~\ref{thm3:characterization}).
Remarkably, it turns out that two lens spaces are $p$-isospectral for every $p$, if and only if the corresponding lattices  are $\norma{\cdot}$-isospectral and satisfy an additional geometric condition: for each $k\in\N$ and $0\leq \zz\leq m$ there are the same number of elements $\mu$  in each lattice having $\norma{\mu} = k$ and exactly $\zz$ coordinates equal to zero.
We call such lattices  $\norma{\cdot}^*$-isospectral.

In Section~\ref{sec:finiteness} we define, for any congruence lattice ${\mathcal L}$, a finite set of numbers $N_{\mathcal L}^{\mathrm{red}}(k,\zz)$ that
count the number of lattice points of a fixed norm $k$ in a small cube, having exactly $\zz$ zero coordinates.
We show that these numbers determine the  $p$-spectrum of the associated lens space.
In particular, one can decide with finitely many computations, whether two lens spaces are $p$-isospectral for all $p$.
We implemented an algorithm in Sage~\cite{Sage} obtaining all examples in dimensions $n=5$ and $7$ for values of $q$ up to $300$ and $150$  respectively (see Tables~\ref{table:m=3} and \ref{table:m=4} in Section~\ref{sec:examples}). We also include some open questions on the nature of the existing examples. We point out that in \cite{DD} the authors make a nice improvement, answering one of our questions.

As a next step we exhibit many pairs of $\norma{\cdot}^*$-isospectral congruence lattices.
We do this in Section~\ref{sec:families} proving that, for any $r\geq7$ and $t$ positive integers such that $r$ is coprime to $3$, the congruence lattices
$$
\mathcal L(r^2t;\;1,\;1+rt,\;1+3rt)
\quad\text{and}\quad
\mathcal L(r^2t;\;1,\;1-rt,\;1-3rt)
$$
are $\norma{\cdot}^*$-isospectral.
In order to prove the equality of $N_{\mathcal L}(k,\zz)$ and $N_{\mathcal L'}(k,\zz)$  for every $k$ and $\zz$ we develop a procedure to compute these numbers.

By using the results in the previous sections, we obtain  the corresponding results for lens spaces in Section~\ref{sec:all-p-iso}.
The pairs of $\norma{\cdot}^*$-isospectral congruence lattices from Section~\ref{sec:families} produce an infinite family of pairs of $5$-dimensional lens spaces that are $p$-isospectral for all $p$.
This allows to obtain lens spaces that are $p$-isospectral for all $p$ in arbitrarily high dimensions by using a result of Ikeda (Theorem~\ref{thm7:high-dim}).
We point out that the resulting  lens spaces are homotopically equivalent but cannot be homeomorphic to each other (see Lemma~\ref{lem:homotequiv}).

It is well known that two non-isometric lens spaces cannot be strongly isospectral (see Proposition~\ref{prop7:lens-non-strongly}).
In particular, isospectral non-isometric lens spaces cannot be constructed by the Sunada method.
In the case of the simplest pair of $5$-dimensional lens spaces $L=L(49;1,6,15)$ and $L'=L(49;1,6,20)$ that are $p$-isospectral for all $p$, we give many representations $\tau$ of $K=\SO(5)$ for which the associated natural strongly elliptic operators do not have the same spectrum (Section~\ref{sec:tau-isospectrality}). Actually, these lens spaces are very far from being strongly isospectral.

The method based on representation theory to characterize $0$-spectrum for a lens space could also be applied to many other spaces, for example, orbifold lens spaces (see Remark~\ref{rem7:orbifolds}), arbitrary spherical space forms and more general locally symmetric spaces of compact type.

\section{Preliminaries}\label{sec:prelim}
Let $G$ be a compact Lie group and let $K$ be a compact subgroup, and let $X=G/K$ endowed with a $G$-invariant Riemannian metric induced by a $G$-biinvariant metric on $G$.
We shall assume that $G$ is semisimple.
Let $\Gamma$ be a discrete subgroup of $G$ that acts freely on $X$, thus the manifold $\Gamma\ba X$ inherits a locally $G$-invariant Riemannian structure.

\subsection{Homogeneous vector bundles}\label{subs:vector-bundles}
For each finite dimensional unitary representation $(\tau, W_\tau)$ of $K$, we consider the \emph{homogeneous vector bundle}
$$
E_\tau= G \times_\tau W_\tau \longrightarrow X=G/K
$$
(see for instance \cite[\S 2.1]{LMRrepequiv}).
We recall that the space $\Gamma^{\infty}(E_\tau )$ of smooth sections of $E_\tau$ is isomorphic to the space $C^\infty(G/K;\tau)$ of smooth functions $C^\infty(G/K;\tau):=\{ f: G \rightarrow W_\tau$ such that $f(xk)= \tau(k^{-1})f(x)\}$.

We form the vector bundle $\Gamma\ba E_\tau$ over the manifold $\Gamma\ba G/K$ and denote it by $L^2(\Gamma\ba E_\tau)$ the closure of $C^\infty (\Gamma\ba G/K;\tau)$ with respect to the inner product
$
(f_1,f_2)= \int_{\Gamma\ba X} \langle f_1(x),f_2(x)\rangle \;\mathrm{d}x$, where $\langle\,\,\rangle$ is a $\tau$-invariant inner product on $W_\tau$.

The complexification $\mathfrak g$  of the Lie algebra $\mathfrak g_0$ of $G$ and the universal enveloping algebra $U(\mathfrak g)$ act on $C^\infty(G/K;\tau)$ by left invariant differential operators in the usual way.
We shall denote by $C=\sum X_i^2$ the \emph{Casimir} element of $\mathfrak g$, where $X_1,\dots,X_n$ is any orthonormal basis of $\mathfrak g$;
$C$ lies in the center of $U(\mathfrak g)$ and defines  second order elliptic \emph{differential operators}
$\Delta_\tau$ on $C^\infty(G/K;\tau)$ and  $\Delta_{\tau,\Gamma}$ on $\Gamma\ba E_\tau$. The Casimir element $C$ acts on an irreducible representation $V_\pi$ of $G$ by a scalar $\lambda(C,\pi)$.

Consider the left regular representation of $G$ on $L^2( E_\tau) \simeq L^2(G/K;\tau)$. By Frobenius reciprocity, the multiplicity of an irreducible representation $\pi$ of $G$ equals $[\tau:\pi_{|K}] : = \dim Hom_K(V_\tau, V_\pi)$. We thus have
\begin{equation}\label{eq2:L^2(G/K, tau)}
L^2(G/K;\tau)=\sum_{\pi\in\widehat G} [\tau:\pi_{|K}]\, V_\pi.
\end{equation}
Thus, by taking $\Gamma$-invariants,
\begin{equation}\label{eq2:L^2(G/K, tau)}
L^2(\Gamma\ba G/K;\tau)=\sum_{\pi\in\widehat G} [\tau:\pi_{|K}]\, V_\pi^\Gamma,
\end{equation}
where $V_\pi^\Gamma$ is the space of $\Gamma$-invariant vectors in $V_\pi$. We set $d_\pi^\Gamma=\dim V_\pi^\Gamma$.

Similarly we may consider the right regular representation of $G$ on
$$L^2(\Gamma \ba G) = \sum_{\pi \in\widehat G} n_{\pi}(\Gamma) \pi\,.$$
By Frobenius reciprocity, we get in this case that $n_{\pi}(\Gamma)= d_\pi^\Gamma=\dim V_\pi^\Gamma$.
Hence we have the decomposition
\begin{equation}\label{eq2:L^2(GammaG)}
L^2(\Gamma\ba G)=\sum_{\pi\in\widehat G} d_\pi^\Gamma\, V_\pi.
\end{equation}

 Hence, taking into account \eqref{eq2:L^2(G/K, tau)} and \eqref{eq2:L^2(GammaG)}  we obtain:

\begin{prop}\label{prop2:tau-equiv=>tau-iso}
Let $(G,K)$ be a symmetric pair of compact type and let $\Gamma$ be a discrete cocompact subgroup of $G$ that acts freely on $X=G/K$.
Let $\Delta_{\tau,\Gamma}$ be the Laplace operator acting on the sections of the homogeneous vector bundle $\Gamma\ba E_\tau$ of the manifold $\Gamma\ba X$.
If $\lambda\in\R$, the multiplicity $d_{\lambda}(\tau,\Gamma)$ of the eigenvalue $\lambda$ of $\Delta_{\tau,\Gamma}$  is given by
\begin{equation}\label{eq2:mult_lambda}
d_{\lambda}(\tau, \G)=
\sum_{\pi\in \widehat G :\, \lambda(C,\pi)=\ld}
d_\pi^\Gamma \;[\tau:\pi_{|K}] .
\end{equation}
\end{prop}

In the case when $\Gamma$ is a finite abelian group inside a maximal torus $T$ of $G$ one can further write the dimension  $d_\pi^\Gamma$
of the space of $\Gamma$-invariants in $V_\pi$ in a simple way in terms of weight multiplicities (see Lemma \ref{lem3:L_Gamma}).

\subsection{Spherical space forms}
In this subsection we will recall the description of the $p$-spectrum of the Hodge-Laplace operator on spherical space forms.
We will restrict our attention to odd dimensions, namely, spaces $\Gamma\ba S^{2m-1}$ where $\Gamma$ is a finite subgroup of $\SO(2m)$ that acts freely on $S^{2m-1}$. We first recall some general facts on the representation theory of compact Lie groups.

We note that if a discrete (finite) subgroup $\Gamma\subset\Ot(n+1)$ acts freely on $S^n$, then it must necessarily be included in $\SO(n+1)$, thus $\Gamma\ba S^n$ is an orientable manifold.

We set $G=\SO(2m)$.
We fix the standard maximal torus in $G$,
\begin{equation*}
T=\left\{t=
\diag\left(
\left[\begin{smallmatrix}\cos(2\pi\theta_1)&-\sin(2\pi\theta_1) \\ \sin(2\pi\theta_1)&\cos(2\pi\theta_1)
\end{smallmatrix}\right]
,\dots,
\left[\begin{smallmatrix}\cos(2\pi\theta_m)&-\sin(2\pi\theta_m) \\ \sin(2\pi\theta_m)&\cos(2\pi\theta_m)
\end{smallmatrix}\right]
\right)
:\theta\in\R^m
\right\}.
\end{equation*}
The Lie algebra of $T$ is given by
\begin{equation}\label{eq2:h_0}
\mathfrak h_0=\left\{ H=
\diag\left(
\left[\begin{smallmatrix}0&-2\pi \theta_1\\ 2\pi \theta_1&0\end{smallmatrix}\right]
, \dots,
\left[\begin{smallmatrix}0&-2\pi \theta_m\\ 2\pi \theta_m&0\end{smallmatrix}\right]
\right)
:\theta\in\R^m
\right\}.
\end{equation}
Note that $t=\exp(H)$ if $t\in T$ and $H\in \mathfrak h_0$ as above.
The Cartan subalgebra $\mathfrak h:=\mathfrak h_0\otimes_\R \C$ is given as in \eqref{eq2:h_0} with $\theta_1,\dots,\theta_m\in\C$, and in this case we let $\varepsilon_j\in\mathfrak h^*$ be given by $\varepsilon_j(H)=2\pi i\theta_j$ for any $1\leq j\leq m$.
The weight lattice for $G=\SO(2m)$ is $P(G)=\bigoplus_{j=1}^m\Z\varepsilon_j$.

We fix the standard system of positive roots  $\Delta^+(\mathfrak g,\mathfrak h)=\{\varepsilon_i\pm\varepsilon_j: 1\leq i<j\leq m\}$, with system of simple roots $\{\varepsilon_j-\varepsilon_{j+1}: 1\leq j\leq m-1\}\cup\{\varepsilon_{m-1}+\varepsilon_m\}$ and  dominant weights of the form $\sum_{j=1}^m a_j\varepsilon_j\in P(G)$ such that $a_1\geq\dots\geq a_{m-1}\geq |a_m|$.

We denote by $\{e_1,\dots,e_{2m}\}$ the standard basis of $\R^{2m}$.
If $K= \{k \in \SO(2m): k e_{2m} = e_{2m}\} \simeq\SO(2m-1)$, then we take the maximal torus $T\cap K$, thus the Cartan subalgebra associated $\mathfrak h_K$ can be seen as included in $\mathfrak h$.
Under this convention, the positive roots are $\{\varepsilon_i\pm\varepsilon_j:1\leq i<j\leq m-1\}\cup\{\varepsilon_i:1\leq i\leq m-1\}$, the simple roots are $\{\varepsilon_j-\varepsilon_{j
+1}: 1\leq j\leq m-2\}\cup \{\varepsilon_{m-1}\}$, the weight lattice of $K$ is $P(K)=\bigoplus_{j=1}^{m-1} \Z \varepsilon_j$ and  $\mu=\sum_{j=1}^{m-1}a_j\varepsilon_j\in P(K)$ is dominant if and only if $a_1\geq\dots \geq a_m\geq0$.

We  consider  on $\mathfrak g=\so(2m,\C)$ the inner product given by $\langle X,Y\rangle=-(2n-2)^{-1}B(X,\theta Y)$, where $B$ is the Killing form and $\theta$ is the Cartan involution.
One can check that $\langle X,Y\rangle=\operatorname{Trace}(XY)$ for $X,Y\in\mathfrak g$, and this inner product induces on $G/K=S^{2m-1}$ the Riemannian metric of constant sectional curvature $1$.
Furthermore, $\{\varepsilon_1,\dots,\varepsilon_m\}$ is an orthonormal basis of $\mathfrak h^*$.

If $\Gamma$ is a finite subgroup of $G$ acting freely on $S^{2m-1}$, denote by $\Delta_{p,\Gamma}$ the Hodge-Laplace operator on $p$-forms on the spherical space form $\Gamma\ba S^{2m-1}$.
That is, $\Delta_{p,\Gamma}=dd^*+d^*d:\bigwedge^p T^* M\to \bigwedge^p T^* M$, where $d$ is the differential, $d^*$ the codifferential and $\bigwedge^p T^* M$ is the $p$-exterior cotangent bundle of $M$.
As usual, the \emph{$p$-spectrum} of $\Gamma\ba S^{2m-1}$ stands for the spectrum of $\Delta_{p,\Gamma}$ and we say that two manifolds are \emph{$p$-isospectral} if their $p$-spectra coincide.
Spherical space forms are always orientable, thus their $p$-spectra  coincide with the $(2m-1-p)$-spectra for all $0\leq p\leq 2m-1$.

We next describe the $p$-spectrum of any odd-dimensional spherical space form $\Gamma\ba S^{2m-1}$ in terms of $\Gamma$-invariants.
We first introduce some more notation.
Let $\mathcal E_0=\{0\}$ and
$$
\mathcal E_p=\{\lambda_{k,p} := k^2 + k(2m-2) + (p-1)(2m-1-p): k\in\N\}
$$
for $1\leq p\leq m$.
A known and useful fact is that $\mathcal E_p$ and $\mathcal E_{p+1}$ are disjoint for every $0\leq p\leq m-1$ (see for instance \cite[Rmk. after Thm.~4.2]{IkTa} \cite[Rmk.~1.14]{Ik88} and \cite[Thm.~1.1]{LMRrepequiv}).
Let
\begin{equation}\label{eq2:Lambda_p}
\Lambda_p = \begin{cases}
0    \quad&\text{if }p=0,\\
\varepsilon_1+ \varepsilon_2+\dots+\varepsilon_{p}
    \quad&\text{if }1 \leq p\leq m.
\end{cases}
\end{equation}
Let $\pi_{k,p}$ denote the irreducible representation of $\SO(2m)$ with highest weight $k\varepsilon_1+\Lambda_p$ for $0\leq p<m$, and let   $\pi_{k,m}$  denote the sum of the irreducible representations with highest weights $k\varepsilon_1+\Lambda_m$ and $k\varepsilon_1+\overline \Lambda_m$, where $\overline \Lambda_m= \varepsilon_1 + \dots +\varepsilon_{m-1} -\varepsilon_m$.
We will usually write $\pi_{k\varepsilon_1}$ and $\pi_{\Lambda_p}$ in place of $\pi_{k,0}$ and $\pi_{0,p}$ respectively.

\begin{prop}\label{prop2:p-spectrum}
Let $\Gamma\ba S^{2m-1}$ be a spherical space form and let $p$ be such that $0\leq p\leq m-1$.
If $\lambda$ is an eigenvalue of $\Delta_{p,\Gamma}$ then $\lambda\in\mathcal E_p\cup\mathcal E_{p+1}$. Its  multiplicity is given by
\begin{align*}
d_\lambda (0,\Gamma) &=
\dim V_{\pi_{k\varepsilon_1}}^\Gamma = n_\Gamma(\pi_{k\varepsilon_1})
    \quad \text{if } \lambda= k^2 + k(2m-2)\in \mathcal E_0\cup\mathcal E_1,
\quad \textrm{ for } p=0, \\
d_\lambda (p,\Gamma) & =
\begin{cases}
\dim V_{\pi_{k,p}}^\Gamma = n_\Gamma(\pi_{k,p})
   &\text{if } \lambda=\lambda_{k,p}\in\mathcal E_p,\\
\dim V_{\pi_{k,p+1}}^\Gamma = n_\Gamma(\pi_{k,p+1})   &\text{if } \lambda=\lambda_{k,p+1}\in\mathcal E_{p+1},
\end{cases}
\quad \textrm{ for } 1\le p \le m-1.
\end{align*}
\end{prop}

When $\Gamma={1}$ this description appears  in \cite{IkTa}, the case for general $\Gamma$ involves only minor modifications (see \cite[Thm.~1.1]{LMRrepequiv}).

The following proposition follows from \cite[Thm.~4.2]{IkTa} (see also \cite[Prop.~2.1]{Ik88}).
It will be a useful tool to prove one of the main results in the next section.
We include a proof for completeness.

\begin{prop}\label{prop2:charact}
Let $\Gamma\ba S^{2m-1}$ and $\Gamma'\ba S^{2m-1}$ be spherical space forms.
Then
\begin{enumerate}\item[(i)]
$\Gamma\ba S^{2m-1}$ and $\Gamma'\ba S^{2m-1}$ are $0$-isospectral if and only if, for every $k\in \N$,
\begin{equation*}
\dim V_{\pi_{k\varepsilon_1}}^{\Gamma} = \dim V_{\pi_{k\varepsilon_1}}^{\Gamma'}.
\end{equation*}
Generally, for any $0\le p\le m-1$, $\Gamma\ba S^{2m-1}$ and $\Gamma'\ba S^{2m-1}$ are $p$-isospectral if and only if, for every $k\in \N$,
\begin{equation*}
\dim V_{\pi_{k,p}}^{\Gamma} = \dim V_{\pi_{k,p}}^{\Gamma'}
\quad\text{and}\quad
\dim V_{\pi_{k,p+1}}^{\Gamma} = \dim V_{\pi_{k,p+1}}^{\Gamma'}.
\end{equation*}
\item[(ii)]
$\Gamma\ba S^{2m-1}$ and $\Gamma'\ba S^{2m-1}$ are $p$-isospectral for every $p$, if and only if
\begin{equation*}
\dim V_{\pi_{k,p}}^{\Gamma} = \dim V_{\pi_{k,p}}^{\Gamma'}
\end{equation*}
for every $1\leq p\leq m$ and every $k\in\N$.
\end{enumerate}
\end{prop}

\begin{proof}
By Proposition~\ref{prop2:p-spectrum}, if $\lambda\in\R$ is an eigenvalue of $\Delta_{p,\Gamma}$ then $\lambda\in \mathcal E_p\cup \mathcal E_{p+1}$ for some $k\in\N$ and its multiplicity is $\dim V_{\pi_{k,p}}^{\Gamma}$ or $\dim V_{\pi_{k,p+1}}^{\Gamma}$ depending on whether $\lambda$ is in $\mathcal E_p$ or in $\mathcal E_{p+1}$.
Since $\mathcal E_p\cap \mathcal E_{p+1}$ is empty, then (i) follows.
Note that $\pi_{k,0}$ and $\pi_{k,1}$ are the irreducible representations of $\SO(2m)$ with highest weight $k\varepsilon_1$ and $(k+1)\varepsilon_1$ respectively.

Item (ii) follows from (i) since $p$-isospectrality for $0\leq p\leq m-1$ implies $p$-isospectrality for every $p$.
\end{proof}

\section{Isospectrality conditions for lens spaces}
This section contains the first main result in this paper that gives a characterization of  pairs of lens spaces that are either $0$-isospectral or  $p$-isospectral for every  $p$ (see Theorem~\ref{thm3:characterization}) in terms of geometric properties of their associated lattices.

Odd dimensional lens spaces can be described as follows: for each $q\in\N$ and $s_1,\dots,s_m\in\Z$ coprime to $q$,  denote
\begin{equation}\label{eq3:L(q;s)}
L(q;s_1,\dots,s_m) = \langle\gamma\rangle \ba S^{2m-1}
\end{equation}
where
\begin{equation}\label{eq3:gamma}
\gamma=
\diag\left(
\left[\begin{smallmatrix}\cos(2\pi{s_1}/q)&-\sin(2\pi{s_1}/q) \\ \sin(2\pi{s_1}/q)&\cos(2\pi{s_1}/q)
\end{smallmatrix}\right]
,\dots,
\left[\begin{smallmatrix}\cos(2\pi{s_m}/q)&-\sin(2\pi{s_m}/q) \\ \sin(2\pi{s_m}/q)&\cos(2\pi{s_m}/q)
\end{smallmatrix}\right]
\right)
\end{equation}
The element $\gamma$ generates a cyclic group of order $q$ in $\SO(2m)$ that acts freely on $S^{2m-1}$.
Sometimes we shall abbreviate $L(q;\bs)$ in place of $L(q; s_1,\dots,s_m)$, where $\bs$ stands for the vector $\bs=(s_1,\dots,s_m)\in\Z^m$.
The following fact is well known (see \cite[Ch.~V]{Co} or \cite[\S12]{Mil2}).

\begin{prop}\label{prop3:lens-isom}
Let $L=L(q;\bs)$ and $L'=L(q;\bs')$ be lens spaces.
Then the following assertions are equivalent.
\begin{enumerate}
  \item $L$ is isometric to $L'$.
  \item $L$ is diffeomorphic to $L'$.
  \item $L$ is homeomorphic to $L'$.
  \item There exist $\sigma$ a permutation of $\{1,\dots,m\}$, $\epsilon_1,\dots,\epsilon_m\in\{\pm1\}$ and $t\in\Z$ coprime to $q$ such that
  $$
  s_{\sigma(j)}'\equiv t\epsilon_js_j\pmod{q}
  $$
  for all $1\leq j\leq m$.
\end{enumerate}
\end{prop}

The next definition will play a main role in the rest of this paper.

\begin{defi}
Let $q\in\N$ and $\bs=(s_1, \ldots, s_m) \in\Z^m$ such that each entry $s_j$ is coprime to $q$.
We associate to the lens space $L(q;\bs)$ the \emph{congruence lattice }
\begin{equation}\label{eq3:Lambda(q;s)}
\mathcal L(q;s_1,\dots,s_m) =\{(a_1,\dots,a_m)\in\Z^m: a_1s_1+\dots+a_ms_m\equiv 0\pmod q\}.
\end{equation}
\end{defi}

For $\mu=(a_1,\dots,a_m)\in\Z^m$, we set $\norma{\mu}=\sum_{j=1}^m |a_j|$.

\begin{prop}\label{prop3:isometrias}
Let $L(q;\bs)$, $L(q;\bs')$ be lens spaces with $\mathcal L(q,\bs)$ and $\mathcal L(q,\bs')$ the associated lattices.
Then, $L(q;\bs)$ and $L(q;\bs')$ are isometric if and only if $\mathcal L(q;\bs)$ and $\mathcal L(q;\bs')$ are $\norma{\cdot}$-isometric.
\end{prop}

\begin{proof}
By Proposition~\ref{prop3:lens-isom}, $L$ and $L'$ are isometric if and only if there exist $t$ coprime to $q$ and $\varphi$, a composition of permutations and changes of signs, such that $\varphi(t\bs)=\varphi(t s_1,\dots, t s_m)= (s_1', \ldots, s'_m)=\bs'$.
Hence $\mathcal L(q,\bs')=\mathcal L(q,\varphi (\bs))=\varphi (\mathcal L(q,\bs))$ with $\varphi$ a $\norma{\cdot}$-isometry.

In order to prove the converse assertion, we first show that every $\norma\cdot$-linear isometry of $\R^n$ is a composition of permutations and changes of signs.
If $T$ is a $\norma{\cdot}$-linear isometry of $\R^n$, then for each $1\le k \le  n$,  $T(\varepsilon_k)= \sum_{j=1}^n c_{k,j} \varepsilon_j$ with $\sum |c_{k,j}| =1$. We claim that $c_{k,j}\ne 0$ for at most one value of $j$. Otherwise, there are $h,k,\ell$ such that $c_{k,\ell}c_{h,\ell}\ne 0$. Hence $|c_{k,\ell} +\delta c_{h,\ell}|< |c_{k,\ell}| + |c_{h,\ell}|$, for $\delta =1$ or $\delta =-1$.
Thus, for this choice of $\delta$ we have $2 = \norma{T(\varepsilon_k) +\delta  T(\varepsilon_h)} = \sum_{j=1}^n |c_{k,j} + \delta c_{k,j}|<
\sum_{j=1}^n |c_{k,j}| + |c_{k,j}|=2$, a contradiction.

Now, suppose conversely that $\varphi$ is a $\norma{\cdot}$-isometry between $\mathcal L(q,s)$ and $\mathcal L(q,s')$.
The previous paragraph ensures that $\varphi$ is given by
$$
\varphi(a_1,\dots,a_m)=(\epsilon_{\sigma(1)} a_{\sigma(1)},\dots, \epsilon_{\sigma(m)} a_{\sigma(m)})
$$
with $\sigma$ a permutation of $\{1,\dots,m\}$ and $\epsilon_j=\pm1$ for all $j$, and satisfies $\mathcal L(q;\bs')= \varphi(\mathcal L(q;\bs))$, thus $\mathcal L(q;\bs')= \mathcal L(q;\varphi(\bs))$.
For each $2\leq j\leq m$, the vector $$(-s_j',0,\dots,0,s_1',0,\dots,0)$$ lies in $\mathcal L(q;\bs')$, thus $-s_j' \epsilon_{\sigma(1)} s_{\sigma(1)} + s_1' \epsilon_{\sigma(j)} s_{\sigma(j)}\equiv 0 \pmod q$ since it is also in $\mathcal L(q;\varphi(\bs))$.
Then, if $t\in\Z$ is such that $ t\epsilon_{\sigma(1)} s_{\sigma(1)}\equiv s_1'\pmod q$, one has that $s_j' \equiv t \epsilon_{\sigma(j)} s_{\sigma(j)} \pmod q$ for every $j$.
Hence $L$ and $L'$ are isometric to each other.
This completes the proof.
\end{proof}

The goal of this section is to write the $p$-spectrum of a lens space in terms of the $\norma{\cdot}$-length spectrum of the associated congruence lattice.
To do this we will express the numbers $\dim V_{\pi_{k,p}}^\Gamma$ in terms of weight multiplicities of representations of $G=\SO(2m)$.
We will identify the weight lattice $P(G) = \bigoplus_{j=1}^m \Z \varepsilon_j$ with $\Z^m$  via the correspondence $\sum_{j=1}^m a_j\varepsilon_j$ $\mapsto$ $(a_1,\dots,a_m)$.

\begin{lemma}\label{lem3:L_Gamma}
Let $\Gamma =\langle \gamma \rangle$  where $\gamma$ is as in \eqref{eq3:gamma}. Let $L = L(q;s_1,\dots,s_m)$ be the corresponding lens space and let $\mathcal L=\mathcal L(q;s_1,\dots,s_m)$ be the associated lattice.
If  $(\pi,V_\pi)$  is a finite dimensional representation of $\SO(2m)$, then
\begin{equation}\label{eq:Gammainvar}
\dim V_\pi^\Gamma=
\sum_{\mu\in \mathcal L}\, m_\pi(\mu),
\end{equation}
where $m_\pi(\mu)$ denotes the multiplicity of the weight $\mu$ in $\pi$.
\end{lemma}

\begin{proof}
One has that $V_\pi=\oplus_{\mu\in P(G)} V_\pi(\mu)$, where $V_\pi(\mu)$ is the $\mu$-weight space, i.e.\ the space of vectors $v$ such that $\pi(h)v=h^\mu v$ for every $h\in T$.
Here, $h^\mu=e^{\mu(X_h)}$ where $X_h$ is any element in $\mathfrak h_0$ satisfying $\exp(X_h)=h$.
Thus, $V_\pi^\Gamma= \oplus_{\mu\in P(G)} V_\pi(\mu)^\Gamma$.
Now, $v\in V_\pi(\mu)$, $v\ne 0$, is $\Gamma$-invariant if and only if $\gamma^\mu=1$, hence
$
\dim V_\pi^\Gamma =\sum_{\mu:\gamma^\mu=1} \dim V_{\pi}(\mu)=\sum_{\mu:\gamma^\mu=1} m_\pi(\mu).
$

We let
$$
H_\gamma=\diag\left( \left(\begin{smallmatrix}0&-2\pi s_1/q\\ 2\pi s_1/q&0\end{smallmatrix}\right) ,\dots, \left(\begin{smallmatrix}0&-2\pi s_m /q\\ 2\pi s_m/q&0\end{smallmatrix}\right) \right),
$$
thus $\exp (H_\gamma)=\gamma$.
If $\mu=\sum_{j=1}^{m} a_j\varepsilon_j\in P(\SO(2m))$ then
$$
\gamma^\mu = e^{\mu(H_\gamma)} = e^{-2\pi i\left(\frac{a_1s_1+\dots+a_ms_m}{q}\right)}=1
$$
if and only if $a_1s_1+\dots+a_ms_m\equiv 0\pmod q$, that is, $\mu \in \mathcal L$.
\end{proof}

Let $\mathcal L$ be an arbitrary sublattice of $\Z^m$. For $\mu\in\Z^m$ we set $Z(\mu)=\#\{j:1\leq j\leq m,\, a_j=0\}$.
We denote, for any $0\leq \zz \leq m$ and any $k\in\N_0:=\N\cup\{0\}$,
\begin{align}
N_\mathcal L(k)  &= \#\left\{\mu\in \mathcal L : \norma{\mu}=k\right\},\label{eq:Nk}\\
N_{\mathcal L}(k,\zz ) &= \#\left\{\mu\in \mathcal L : \norma{\mu}=k,\; Z(\mu)=\zz \right\}.\label{eq:Nkz}
\end{align}

\begin{defi}\label{def3:isosp-latt}
Let $\mathcal L$ and $\mathcal L'$ be sublattices of $\Z^m$.
\begin{enumerate}
\item[(i)] $\mathcal L$ and $\mathcal L'$ are said to be \emph{$\norma{\cdot}$-isospectral} if $N_{\mathcal L}(k) = N_{\mathcal L'}(k)$ for every $k\in\N$.
\item[(ii)]
 $\mathcal L$ and $\mathcal L'$ are said to be \emph{$\norma{\cdot}^*$-isospectral} if $N_{\mathcal L} (k,\zz ) = N_{\mathcal L'}(k,\zz )$ for every $k\in\N$ and every $0\leq \zz \leq m$.
\end{enumerate}
\end{defi}

We will need two useful lemmas on weight multiplicities.
The first one follows from well known facts, but we could not find it stated in the form below, so we include a short proof here.
Recall that $\Lambda_p$ is given by \eqref{eq2:Lambda_p} and $ \pi_{0,p} = \pi_{\Lambda_p}$ is the exterior representation of $\SO(2m)$ on $\bigwedge^p \C^{2m}$ for $0 \le p \le m$.

\begin{lemma}\label{lem3:mult-k=1-p=1}
Let $k\in\N$ and $0\leq p\leq m$.
If $\mu=\sum_{j=1}^m a_j\varepsilon_m\in\Z^m$ we have
\begin{align}
m_{\pi_{k\varepsilon_1}}(\mu) =&
\begin{cases}
\binom{r+m-2}{m-2} & \text{ if }\, \norma{\mu}=k-2r \;\text{ with } r\in \N_0,\\
0 & \text{ otherwise,}
\end{cases}
\label{eq3:mult-pi_k,1} \\
m_{\pi_{\Lambda_{p}}}(\mu) =&
\begin{cases}
\binom{m-p+2r}{r} & \text{if }\,\norma{\mu}=p-2r  \;\text{ with } r\in \N_0,  \text{ and }  |a_j|\leq1\;\forall\,j,\\
0&\text{otherwise.}
\end{cases}
\label{eq3:mult-pi_1,p}
\end{align}
\end{lemma}

\begin{proof}
It is well known that the representation $\pi_{k\varepsilon_1}$ can be realized in the space of harmonic homogeneous polynomials $\mathcal H_k$ of degree $k$ in $m$ variables.
Moreover, $\mathcal P_k\simeq\mathcal H_k\oplus\mathcal P_{k-2}$ where $\mathcal P_k$ denotes the space of homogeneous polynomials of degree $k$, thus
\begin{equation}\label{eq3:H_k=P_k+r^2P_k-2}
m_{\pi_{k\varepsilon_1}}(\mu) = m_{{\mathcal P}_k}(\mu) - m_{{\mathcal P}_{k-2}}(\mu).
\end{equation}

In order to find the weights of $\mathcal P_k$, we set $f_j(x)=x_{2j-1}+ix_{2j},\, f_{j+m}=x_{2j-1}-ix_{2j}\in\mathcal P_1$ for each $1\leq j\leq m$.
It can be easily seen that the polynomials $f_1^{l_1}\dots f_{2m}^{l_{2m}}$ with $\sum_{j=1}^{2m} l_j=k$ form a basis of $\mathcal P_k$ given by weight vectors.
Indeed, $h\in T$ acts on  $f_1^{l_1}\dots f_{2m}^{l_{2m}}$ by multiplication by $h^\mu$ where $\mu=\sum_{j=1}^m (l_j-l_{j+m})\varepsilon_j$.
It follows that   $\mu=\sum_{j=1}^m a_j\varepsilon_j\in\Z^m$ is a weight of $\mathcal P_k$ if and only if there are $l_1,\dots,l_{2m}\in\N_0$ such that $a_j=l_j-l_{j+m}$ and $\sum_{j=1}^{2m} l_j=k$.
Furthermore, one checks that the last condition is equivalent to $k-\norma{\mu}=2r$ with $r\in \N_0$. Hence, $m_{\mathcal P_k}(\mu)$ equals the number of different ways one can write $r$ as an ordered sum of $m$ different nonnegative integers, which equals $\binom{r+m-1}{m-1}$.
This implies that
\begin{equation*}\label{eq3:mult-P_k}
m_{\mathcal P_k}(\mu) =
\begin{cases}
\binom{r+m-1}{m-1} &\text{ if }r=\frac12(k-\norma{\mu})\in\N_0,\\
0 & \text{ otherwise.}
\end{cases}
\end{equation*}
This formula and \eqref{eq3:H_k=P_k+r^2P_k-2} imply \eqref{eq3:mult-pi_k,1}.

We now prove the second assertion.
The representation $\pi_{\Lambda_p}$ can be realized as  the complexified $p$-exterior representation $\bigwedge^p(\C^{2m})$ with the canonical action  of $\SO(2m)$.
Let $\{e_1,\dots,e_{2m}\}$ denote the canonical basis of $\C^{2m}$.
For $1\leq j\leq m$, we set $v_j=e_{2j-1}-i e_{2j}$ and $v_{j+m}=e_{2j-1}+i e_{2j}$.
Hence $\{v_1,\dots,v_{2m}\}$ is also a basis of $\C^{2m}$ and
\begin{equation}\label{eq5:base_p-forms}
\left\{v_{i_1}\wedge\dots\wedge v_{i_p}: 1\leq i_1<i_2<\dots<i_p\leq 2m\right\}
\end{equation}
is a basis of $\bigwedge^p(\C^{2m})$.
For $I=\{1\leq i_1<i_2<\dots<i_p\leq 2m\}$ we write $\omega_I=v_{i_1}\wedge\dots\wedge v_{i_p}$.

One can check that $h\in T$ acts on $\omega_I$ by multiplication by $h^\mu$ where
$\mu=\sum_{j=1}^m a_j\varepsilon_j$ is given by
$$
a_j=\begin{cases}
 1&\quad\text{if $j\in I$ and $j+m\notin I$,}\\
-1&\quad\text{if $j\notin I$ and $j+m\in I$,}\\
0&\quad\text{if both $j,j+m\in I$, or $j,j+m\notin I$.}
\end{cases}
$$
Thus, an arbitrary element $\mu=\sum_j a_j\varepsilon_j\in\Z^m$ is a weight of $\bigwedge^p(\C^{2m})$ if and only if $|a_j|\leq 1$ for all $j$ and $p-\norma{\mu}\in2\N_0$.

Let $\mu=\sum_{j=1}^m a_j\varepsilon_j\in\Z^m$ be such that $|a_j|\leq 1$ for all $j$ and $r=\frac12(p-\norma{\mu})\in\N_0$.
Let $I_\mu=\{i:1\leq i\leq m, \,  a_i=1\}\cup\{i:m+1\leq i\leq 2m,\, a_{i-m}=-1\}$.
Thus $I_\mu$ has $p-2r$ elements.
It is a simple matter to check that $\omega_I$ is a weight vector with weight $\mu$ if and only if $I$ has $p$ elements, $I_\mu\subset I$ and $I$ has the property  that $j\in I\smallsetminus I_\mu \iff j+m\in I\smallsetminus I_\mu$ for $1\leq j\leq m$.
One can check that there are $\binom{m-p+2r}{r}$ choices for $I$, hence the claim follows.
\end{proof}

The second lemma is crucial in the proof of Theorem~\ref{thm3:characterization}~(ii).
We recall that $\pi_{k,p}$ is the irreducible representation of $\SO(2m)$ with highest weight $k\varepsilon_1+\Lambda_p$ if $p<m$ and, when $p=m$, the sum of the irreducible representations with highest weights $k\varepsilon_1+\Lambda_m$ and $k\varepsilon_1+\overline\Lambda_m$.

\begin{lemma}\label{lem3:mult-k-p}
Let $\mu,\mu'\in P(\SO(2m))\simeq \Z^m$.
If $\norma{\mu}=\norma{\mu'}$ and $Z(\mu)=Z(\mu')$ then $m_{\pi_{k,p}}(\mu) = m_{\pi_{k,p}}(\mu')$ for every $k\in\N$ and every $1\leq p\leq m$.
\end{lemma}

\begin{proof}
We say that a finite dimensional representation $\sigma$ of $\SO(2m)$ \emph{satisfies condition $(\star)$} if $m_\sigma(\mu)=m_\sigma(\mu')$ for every $\mu$ and $\mu'$ such that $\norma{\mu}=\norma{\mu'}$ and $Z(\mu)=Z(\mu')$.
We see, by Lemma~\ref{lem3:mult-k=1-p=1}, that $\pi_{k\varepsilon_1}$ and $\pi_{\Lambda_p}$ satisfy $(\star)$ for every $k$ and $p$.

Next we show that $\sigma := \pi_{k\varepsilon_1} \otimes \pi_{\Lambda_p}$ also satisfies $(\star)$.
Let $\mu=\sum_{i=1}^m a_i \varepsilon_i$ and $\mu'=\sum_{i=1}^m a_i' \varepsilon_i$ in $\Z^m$ be such that $\norma{\mu}=\norma{\mu'}$ and $Z(\mu)=Z(\mu')$.
We fix a bijection $\varrho:[1,m]\to [1,m]$ so that $a_i'\neq0$ if and only if $a_{\varrho(i)}\neq0$.
We have that
\begin{equation}\label{eq3:m_sigma}
m_\sigma(\mu)
=\sum_{\eta\in \Z^m} m_{\pi_{\Lambda_p}}(\eta) \;m_{\pi_{k\varepsilon_1}}(\mu-\eta)
\end{equation}
and a similar expression  for $m_\sigma(\mu')$ (see for instance \cite[Ex.~V.14]{Knapp}).
Both sums are already over the weights of $\pi_{\Lambda_p}$, that is, over the weights  $\eta=\sum_{i=1}^m b_i \varepsilon_i$ such that $|b_i|\leq 1$ for all $i$ and $\norma{\eta}=p-2r$ for some $r\in\N$, by Lemma~\ref{lem3:mult-k=1-p=1}.
To each such $\eta$ we associate $\eta'=\sum_{i=1}^m b_i' \varepsilon_i$ defined by $b_i'=b_{\varrho(i)}$ for every $i$ such that $a_i'=0$ and $b_i'= \textrm{sg}(a_{\varrho(i)})\, \textrm{sg}({a_i'}) \,  b_{\varrho(i)}$ for every $i$ such that $a_i'\neq 0$.
One can check that $\norma{\eta}=\norma{\eta'}$, $Z(\eta) =Z(\eta')$ and furthermore $\norma{\mu-\eta}=\norma{\mu'-\eta'}$, thus $m_{\pi_{\Lambda_p}}(\eta) \, m_{\pi_{k\varepsilon_1}}(\mu-\eta) = m_{\pi_{\Lambda_p}}(\eta') \, m_{\pi_{k\varepsilon_1}}(\mu'-\eta')$.
By \eqref{eq3:m_sigma} we have that $m_\sigma(\mu) = m_\sigma(\mu')$ as asserted.

By Steinberg's formula (see for instance \cite[Ex.~17-Ch.~ IX]{Knapp}), the representation $\sigma$ decomposes as
\begin{equation}\label{eq3:descomp-sigma}
\chi_\sigma
= \sum_\mu m_{\pi_{\Lambda_p}}(\mu) \, \mathrm{sgn}(\mu+k\varepsilon_1+\rho) \, \chi_{(\mu+k\varepsilon_1+\rho)^\vee-\rho},
\end{equation}
where $\chi_\sigma$ denotes the character of the representation $\sigma$, $\rho=\sum_{j=1}^m (m-j)\varepsilon_j$, half the sum of positive roots, $\eta^\vee$ denotes the only dominant weight in the same Weyl orbit as $\eta$, and
$$
\mathrm{sgn}(\mu) =
\begin{cases}
0 &\text{ if $\omega \mu=\mu$ for some nontrivial $\omega\in W$,}\\
\mathrm{sgn}(\omega) &\text{ otherwise, where $\omega\mu$ is dominant}.
\end{cases}
$$

Note that the sum in \eqref{eq3:descomp-sigma} is over the weights of $\pi_{\Lambda_p}$, described in \eqref{eq3:mult-pi_1,p}.
Moreover, the character of the representation $\pi_{k,p}$ appears in the sum on the right-hand side in \eqref{eq3:descomp-sigma} and this is the only time it does, hence $\pi_{k,p}$ appears exactly once in the decomposition of $\sigma$.
Now the proof of the lemma  is completed by an inductive argument in $k$ and $p$ by checking that any other irreducible representation $\pi_{k',p'}$ that appears in \eqref{eq3:descomp-sigma} satisfies $k'<k$, or else $k'=k$ and $p'<p$, thus $\pi_{k',p'}$ satisfies $(\star)$ by the strong inductive hypothesis.
Finally, since $\sigma$ also satisfies $(\star)$ then $\pi_{k,p}$ also does.
\end{proof}

The next theorem gives an explicit formula for $\dim V_{\pi_{k,p}}^\Gamma$ in terms of weight multiplicities $m_{\pi_{k,p}}(\mu)$ and of the numbers   $N_{\mathcal L} (k,\zz )$, when $L=\Gamma\ba S^{2m-1}$ is a lens space with  congruence lattice $\mathcal L$.

\begin{thm}\label{thm3:dim V_k,p^Gamma}
Let $L=\Gamma\ba S^{2m-1}$ be a lens space with associated lattice $\mathcal L$ and let $k\in\N$ and $0\leq p\leq m$.
Then
\begin{equation}\label{eq3:dim V_k,p^Gamma}
\dim V_{\pi_{k,p}}^{\Gamma}=
    \sum_{r=0}^{\lfloor(k+p)/2\rfloor} \;\sum_{\zz =0}^m \; m_{\pi_{k,p}}(\mu_{r,\zz })\; N_{\mathcal L}(k+p-2r,\zz ),
\end{equation}
where $\mu_{r,\zz }$ is any weight such that $Z(\mu_{r,\zz })=\zz $ and $\norma{\mu_{r,\zz }} = k+p-2r$.

In the particular case when $p=0$ we have that
\begin{equation}\label{eq3:dim V_k,1^Gamma}
\dim V_{\pi_{k\varepsilon_1}}^{\Gamma}=
\sum_{r=0}^{\lfloor k/2\rfloor}\binom{r+m-2}{m-2} N_{\mathcal L}(k-2r).
\end{equation}
\end{thm}

\begin{proof}
By Lemma~\ref{lem3:L_Gamma} we have that
\begin{equation*}
\dim V_{\pi_{k,p}}^{\Gamma} = \sum_{\mu\in \mathcal L} m_{\pi_{k,p}}(\mu).
\end{equation*}
The sum is finite since it is a sum over the weights $\mu$ of $\pi_{k,p}$.
These weights are of the form $k\varepsilon_1+\Lambda_{p}-\nu$ with $\nu$ a sum of positive roots, if $p<m$, and of the form $k\varepsilon_1+\Lambda_{m}-\nu$ or $k\varepsilon_1+ \overline \Lambda_{m}-\nu$, if $p=m$.
Since $\norma{\alpha}=\norma{\varepsilon_i\pm\varepsilon_j}=2$ for every positive root $\alpha$ of $\,\mathfrak{so}(2m,\C)$ (see Section~\ref{sec:prelim}), then  $m_{\pi_{k,p}}(\mu)=0$ unless $\norma{k\varepsilon_1+\Lambda_{p}}-\norma{\mu}=k+p-\norma{\mu}\in2\N_0$.
Hence
\begin{align*}
\dim V_{\pi_{k,p}}^{\Gamma}
    &=
    \sum_{r=0}^{\lfloor (k+p)/2\rfloor} \;\sum_{\zz =0}^m \;\sum_{\mu\in \mathcal L:\, Z(\mu)=\zz , \atop  \norma{\mu}=k+p-2r} m_{\pi_{k,p}}(\mu).
\end{align*}
Since, by Lemma~\ref{lem3:mult-k-p}, the value of $m_{\pi_{k,p}}(\mu)$ depends only on $\norma{\mu}$ and $Z(\mu)$, the last sum equals the number of weights $\mu$ such that $\norma{\mu}=k+p-2r$ and $Z(\mu)=\zz $, times the multiplicity of any such weight.
This proves \eqref{eq3:dim V_k,p^Gamma}.

In the case when $p=0$, the multiplicity $m_{\pi_{k\varepsilon_1}}(\mu)$ is as  given in \eqref{eq3:mult-pi_k,1}.
Thus
\begin{align*}
\dim V_{\pi_{k\varepsilon_1}}^{\Gamma}
    &= \sum_{r=0}^{\lfloor k/2\rfloor} \sum_{\mu\in \mathcal L:\atop \norma{\mu}=k-2r} \binom{r+m-2}{m-2}
    = \sum_{r=0}^{\lfloor k/2\rfloor}\binom{r+m-2}{m-2} N_{\mathcal L}(k-2r). \notag
\end{align*}
This completes the proof.
\end{proof}

We now state the first main result in this paper.

\begin{thm}\label{thm3:characterization}
Let $L=\Gamma\ba S^{2m-1}$ and $L'=\Gamma'\ba S^{2m-1}$ be lens spaces with associated congruence lattices $\mathcal L$ and $\mathcal L'$ respectively.
Then
\begin{enumerate}
  \item[(i)] $L$ and $L'$ are $0$-isospectral if and only if $\mathcal L$ and $\mathcal L'$ are $\norma{\cdot}$-isospectral.
  \item[(ii)] $L$ and $L'$ are $p$-isospectral for all $p$ if and only if $\mathcal L$ and $\mathcal L'$ are    $\norma{\cdot}^*$-isospectral.
\end{enumerate}
\end{thm}

\begin{proof}
Proposition~\ref{prop2:charact}~(i) (resp.\ (ii)) says that $L$ and $L'$ are $0$-isospectral (resp.\ $p$-isospectral for all $p$) if and only if,  for every $k\in\N$, $\dim V_{\pi_{k\varepsilon_1}}^{\Gamma} = \dim V_{\pi_{k\varepsilon_1}}^{\Gamma'}$ (resp. $\dim V_{\pi_{k,p}}^{\Gamma} = \dim V_{\pi_{k,p}}^{\Gamma'}$ for every $k\in\N$ and every $1\leq p\leq m$).
Hence, in the converse direction, (i) and (ii) follow immediately from \eqref{eq3:dim V_k,1^Gamma} and \eqref{eq3:dim V_k,p^Gamma} respectively.

We now assume that $L$ and $L'$ are $0$-isospectral.
We shall prove by induction that
\begin{equation}\label{eq3:N(k)=N'(k)}
N_{\mathcal L}(k) = N_{\mathcal L'}(k)
\end{equation}
for every $k \in\N$.
The case $k=0$ is clear, since both sides are equal to one.
Suppose that \eqref{eq3:N(k)=N'(k)} holds for every $k<k_0$.
By \eqref{eq3:dim V_k,1^Gamma} we have that
$$
\sum_{r\geq0} \binom{r+m-2}{m-2} N_{\mathcal L}(k_0-2r)
=
\sum_{r\geq0} \binom{r+m-2}{m-2} N_{\mathcal L'}(k_0-2r).
$$
All the terms with $r>0$ on both sides are equal by assumption, hence this equality implies that also $N_{\mathcal L}(k_0)=N_{\mathcal L'}(k_0)$. This proves (i).

We next prove (ii).
Assume that $L$ and $L'$ are $p$-isospectral for all $p$.
We shall prove that
\begin{equation}\label{eq3:N(k,z)=N'(k,z)}
N_{\mathcal L}(k,\zz ) = N_{\mathcal L'}(k,\zz )
\qquad\forall\zz:
0\leq\zz \leq m,
\end{equation}
for every $k\in\N$.
We use an inductive argument on $k$.
The case $k=0$ is again clear.
We suppose that \eqref{eq3:N(k,z)=N'(k,z)} holds for every $k<k_0$.
For each $1\leq p\leq m$, if we let $k=k_0-p$, then, by \eqref{eq3:dim V_k,p^Gamma}, since $L$ and $L'$ are $p$-isospectral, we have that
$$
\sum_{r\geq0} \sum_{\zz =0}^m \; m_{\pi_{k,p}}(\mu_{r,\zz })\; N_{\mathcal L}(k_0-2r,\zz )
=
\sum_{r\geq0} \sum_{\zz =0}^m \; m_{\pi_{k,p}}(\mu_{r,\zz })\; N_{\mathcal L'}(k_0-2r,\zz ),
$$
where $\mu_{r,\zz }$ is any weight satisfying $\norma{\mu_{r,\zz }}=k_0-2r$ and $Z(\mu_{r,\zz })=\zz $.
By assumption, all terms in both sides with $r>0$ coincide. Thus
\begin{equation*}
\sum_{\zz =0}^{m-1} \; m_{\pi_{k,p}}(\mu_{0,\zz })\; N_{\mathcal L}(k_0,\zz )
=
\sum_{\zz =0}^{m-1} \; m_{\pi_{k,p}}(\mu_{0,\zz })\; N_{\mathcal L'}(k_0,\zz ).
\end{equation*}
Note that the terms $\zz =m$ in both sides have been deleted since they are both  equal to zero.

To prove our claim it suffices to show that the $m\times m$-matrix $(m_{\pi_{k,p}}(\mu_{0,\zz }) )_{p,\zz}$ with $p=1,\dots,m$ and $\zz =0,\dots, m-1$ is invertible. We claim that this matrix  has $1$'s on the anti-diagonal and it is `upper-triangular' with respect to the anti-diagonal, hence it has determinant $\pm 1$.

Now, the element $\mu_{0,\zz }$ is any weight in $\Z^m$ such that $\norma{\mu_{0,\zz }}=k_0$ and $Z(\mu_{r,\zz })=\zz $, thus we may pick
\begin{equation*}
\mu_{0,\zz } = (k_0-m+\zz +1)\varepsilon_1 + \varepsilon_2 + \dots + \varepsilon_{m-\zz }.
\end{equation*}
If $m-\zz=p$ (i.e.\ $(p,\zz)$ is on the antidiagonal), then $\mu_{0,\zz}=k\varepsilon_1+\Lambda_p$.
If $p<m$, then $\pi_{k,p}$ has highest weight $k\varepsilon_1+\Lambda_p$, hence $m_{\pi_{k,p}}(\mu_{0,\zz })=1$.
On the other hand, if $m-\zz<p$ then  $\mu_{0,\zz } $ cannot be a weight since
$k\varepsilon_1 +\Lambda_p - \mu_{0,\zz }$ is not a sum of positive roots  given that the coefficient of $\varepsilon_1$ equals $m-\zz-p<0$.
The case $p=m$ is very similar and its verification is left to the reader.
\end{proof}

\begin{rem}\label{rem3:non-strongly}
Two spherical space forms $\Gamma\ba S^{n}$ and $\Gamma'\ba S^{n}$ are said to be \emph{strongly isospectral} if for any strongly elliptic natural operator $D$ acting on sections of a natural bundle $E$ over $S^n$, the associated operators $D_\Gamma$ and $D_{\Gamma'}$ acting on sections of the bundles $\Gamma\ba E$ and $\Gamma'\ba E$ have the same spectrum.
Isospectral manifolds constructed by Sunada's method are strongly isospectral.
It is a well known fact that non-isometric lens spaces cannot be strongly isospectral (see Proposition~\ref{prop7:lens-non-strongly}).
\end{rem}

\begin{rem}\label{rem3:ej-Ikeda}
Ikeda in \cite{Ik80} gave many pairs of non-isometric lens spaces that are $0$-isospectral.
The simplest such pair is $L(11;1,2,3)$ and $L(11;1,2,4)$ in dimension $5$.
In light of Theorem~\ref{thm3:characterization}~(i), the associated congruence $3$-dimensional lattices $\mathcal L=\mathcal L(11;1,2,3)$ and $\mathcal L'=\mathcal L(11;1,2,4)$ must be $\norma{\cdot}$-isospectral.
However, it is a simple matter to check that $\mathcal L$ and $\mathcal L'$  are not $\norma{\cdot}^*$-isospectral.
In fact, it is easy to see that $\pm(2,-1,0)$ and $\pm(1,1-1)$ are the only vectors in $\mathcal L$ with $1$-norm equal to $3$, while $\pm(2,-1,0)$ and $\pm(0,2,-1)$ are those with $1$-norm equal to $3$ lying in $\mathcal L'$.
This implies that $N_{\mathcal L}(3,0)=2 \neq N_{\mathcal L'}(3,0)=0$ and  $N_{\mathcal L}(3,1)=2 \neq N_{\mathcal L'}(3,1)=4$, proving the assertion.

As we shall see in Section~\ref{sec:families}, there exist infinitely many pairs of congruence lattices that are $\norma{\cdot}^*$-isospectral in dimension $m=3$.
Such examples do not exist for $m=2$, since Ikeda and Yamamoto showed that two $0$-isospectral $3$-dimensional lens spaces are isometric (\cite{IY}, \cite{Y}).
Also, in the relevant paper \cite{Ik88}, Ikeda produced for each given $p_0$ pairs of lens spaces that are $p$-isospectral for every $0\le p \le p_0$ but are not $p_0+1$ isospectral.
\end{rem}

\section{Finiteness conditions}\label{sec:finiteness}

In this section we give a necessary and sufficient condition for two $m$-dimensional $q$-congruence lattices to be $\norma{\cdot}^*$-isospectral, by comparison, for the two lattices, of a finite set of numbers of cardinality at most  $\binom{m+1}{2}q$.
Thus, in light of the connection with lens spaces in Theorem~\ref{thm3:characterization}~(ii), one can check with finitely many computations whether two lens spaces are $p$-isospectral for all $p$.
In Section~\ref{sec:examples} we will show many examples of $\norma{\cdot}^*$-isospectral lattices found with a computer.

We first need to introduce some notions and notations.
For $q\in\N$ we set
$$
C(q)=\left\{\textstyle\sum_j a_j\varepsilon_j\in\Z^m : |a_j|<q\;,\,\forall\,j\right\}.
$$
An element in $C(q)$ will be called \emph{$q$-reduced}.
We define an equivalence relation in $\Z^m$ as follows: if  $\mu=\sum_j a_j\varepsilon_j$, $\mu'=\sum_j a_j'\varepsilon_j\in \Z^m$ then $\mu\sim \mu'$ if and only if $\mu-\mu'\in (q\Z)^m$ and $a_j a_j'\geq 0$ for every $j$ such that $ a_j\not \equiv 0 \pmod q$.
This relation induces an equivalence relation in $\Z^m$ and also in any $q$-congruence lattice $\mathcal L$ since $q\Z^m\subset\mathcal L$.
Furthermore, $C(q)$ and $C(q)\cap\mathcal L$ give a complete set of representatives of $\sim$ on $\Z^m$ and $\mathcal L$ respectively.
We now consider the number of $q$-reduced elements $\mathcal L$ with a fixed norm and a fixed number of zeros.

\begin{defi}
Let $\mathcal L$ be a $q$-congruence lattice as in \eqref{eq3:Lambda(q;s)}.
For any $k\in\N_0$ and $0\leq \zz \leq m$, we set
\begin{equation*}
N_{\mathcal L}^{\mathrm{red}}(k,\zz )
  = \#\{\mu\in C(q)\cap \mathcal L: \norma{\mu} = k,\; Z(\mu)=\zz \}.
\end{equation*}
\end{defi}
We note that any element $\mu\in \Z^m$ lying in the regular tetrahedron  $\norma{\mu}<q$ is $q$-reduced, thus $N_{\mathcal L}(k,\zz ) = N_{\mathcal L}^{\mathrm{red}}(k,\zz )$ for every $k<q$. Also, for each of the $m-\zz$ nonzero coordinates $a_i$ of a $q$-reduced element one has $|a_i|\le q-1$, thus
$N_{\mathcal L}^{\mathrm{red}}(k,\zz )=0$ for every $k>(m-\zz)(q-1)$. Hence, the total number of possibly nonzero numbers $N_{\mathcal L}^{\mathrm{red}}(k,\zz )$ is at most $\binom{m+1}{2}q$.

We have mentioned above that every element in a $q$-congruence lattice $\mathcal L$ is equivalent to one and only one $q$-reduced element in $\mathcal L$.
As one should expect, the finite set of $N_{\mathcal L}^{\mathrm{red}}(k,\zz)$'s determines the numbers $N_{\mathcal L}(k,\zz)$, for every $k, \zz$. That is, if $N_{\mathcal L}^{\mathrm{red}}(k,\zz)=N_{\mathcal L'}^{\mathrm{red}}(k,\zz)$ for every $k$ and $\zz$, then $\mathcal L$ and $\mathcal L'$ are $\norma{\cdot}^*$-isospectral.
The next theorem shows this fact by giving an explicit formula for $N_{\mathcal L}(k,\zz)$ in terms of the $N_{\mathcal L}^{\mathrm{red}}(k,\zz)$. This formula  will also allow us to prove  that the numbers $N_{\mathcal L}(k,\zz)$ determine the numbers $N_{\mathcal L}^{\mathrm{red}}(k,\zz)$.

\begin{thm}\label{thm4:finitud}
Let $\mathcal L$ and $\mathcal L'$ be two $q$-congruence lattices as in \eqref{eq3:Lambda(q;s)}.
\begin{enumerate}
\item[(i)]
If $k=\alpha q+r\in \N$ with $0\leq r < q$, then
\begin{equation}\label{eq4:N^*-N^red}
N_{\mathcal L}(k,\zz )= \sum_{s=0}^{m-\zz } 2^s\binom{\zz +s}{s} \sum_{t=s}^{\alpha}
\binom{t-s+m-\zz -1}{m-\zz -1} \,
N_{\mathcal L}^{\mathrm{red}}(k-tq,\zz +s).
\end{equation}

\item[(ii)]
$N_{\mathcal L}(k,\zz ) = N_{\mathcal L'}(k,\zz )$ for every $k$ and $\zz $ if and only if $N_{\mathcal L}^{\mathrm{red}}(k,\zz ) = N_{\mathcal L'}^{\mathrm{red}}(k,\zz )$ for every $k$ and $\zz $.
\end{enumerate}
\end{thm}

\begin{proof}
We begin by proving (i).
We fix $0\leq r<q$ and we write $k=\alpha q+r$ for some $\alpha \in\N_0$.
When $\alpha=0$ \eqref{eq4:N^*-N^red} is reduced to the identity $N_{\mathcal L}(r,\zz ) = N_{\mathcal L}^{\mathrm{red}}(r,\zz )$, which is valid.
For convenience, in the rest of this proof, we  say that $\mu$ is \emph{of type} $(k,\zz )$ if $\norma{\mu}=k$ and $Z(\mu)=\zz $.

Now assume  that $\alpha=1$.
In this case \eqref{eq4:N^*-N^red} is reduced to
$$
N_{\mathcal L}(q+r,\zz ) =
    N_{\mathcal L}^{\mathrm{red}}(q+r,\zz ) +
    (m-\zz ) N_{\mathcal L}^{\mathrm{red}}(r,\zz ) +
    2(\zz +1) N_{\mathcal L}^{\mathrm{red}}(r,\zz +1).
$$
There are three terms in the right hand side.
Also, if $\mu$ is an element of type $(q+r,\zz )$ and $\mu_0$ is the only element in $C(q)$ such that $\mu\sim\mu_0$, then there are three possible different types for $\mu_0$, namely $(q+r,\zz )$, $(r,\zz )$ and $(r,\zz +1)$.
Next, we check the correspondence between the three terms and the three types, in the same order that are given.

The first term corresponds to the elements in $\mathcal L$ of type $(q+r,\zz )$ which are already reduced.
The second term corresponds to the elements in $\mathcal L$ that are equivalent to a reduced element of type $(r,\zz )$.
Indeed, if $\mu=\sum_i a_i\varepsilon_i \in\mathcal L\cap C(q)$ is of type $(r,\zz )$, then for each nonzero coordinate $i$ of $\mu$ (there are $m-\zz $ of them), the element $\mu+\frac{a_i}{|a_i|}q\varepsilon_i$ has type $(q+r,\zz )$ and lies in the lattice, since $\pm q\varepsilon_i\in q\Z^m \subset \mathcal L$.
Regarding the third term, for each $\mu\in\mathcal L\cap C(q)$ of type $(r,\zz +1)$ and each zero coordinate $i$ of $\mu$ (there are $\zz+1$ of them), the element $\mu\pm q\varepsilon_i$ has type $(q+r,\zz )$.

The detailed description done in the particular case $\alpha=1$ will help to  understand the general case.
Let $\mu\in\mathcal L$ of type $(k,\zz )$ and denote by $\mu_0$ the only element in $C(q)\cap\mathcal L$ such that $\mu\sim\mu_0$. One can check that $\mu_0$ is of type $(k-tq,\zz +s)$ for some $0\leq s\leq m-\zz $ and some $s\leq t\leq \alpha$.

Assume that $\mu_0$ is of type $(k-tq,\zz +s)$.
For each choice of $s$ zero coordinates, $i_1,\dots,i_s$,  of $\mu_0$,  the element $\mu_1:=\mu_0 \pm q \varepsilon_{i_1}\pm\dots\pm  q \varepsilon_{i_s}$ has type $(k-tq+sq,\zz )$.
There are $2^s\binom{\zz +s}{s}$ different ways to choose $\mu_1$ from $\mu_0$.
Now, it remains to add $\pm q$ (depending on the sign of the coordinate)  $(t-s)$-times  in the $m-\zz $ nonzero coordinates.
This can be done in as many ways as  the number of ordered partitions of $t-s$ into $m-\zz$ parts, that is, the number of ways of writing $t-s\in\N_0$ as a sum of $m-\zz $ non-negative integers. This equals $\binom{t-s+m-\zz -1}{m-\zz -1}$ and  establishes formula \eqref{eq4:N^*-N^red}.

We next prove (ii). In one direction the assertion follows from \eqref{eq4:N^*-N^red}.
We now assume that $N_{\mathcal L}(k,\zz ) = N_{\mathcal L'}(k,\zz )$ for every $k$ and $\zz $.
We write $k=\alpha q+r$ with $0\leq r<q$.
We argue by induction on $\alpha$.
When $\alpha =0$, $N_{\mathcal L}(k,\zz ) = N_{\mathcal L}^{\mathrm{red}}(k,\zz )$ and similarly for $\mathcal L'$, thus $N_{\mathcal L}^{\mathrm{red}}(k,\zz ) = N_{\mathcal L'}^{\mathrm{red}}(k,\zz )$ for every $k<q$.

We assume that $N_{\mathcal L}^{\mathrm{red}}(k,\zz ) = N_{\mathcal L'}^{\mathrm{red}}(k,\zz )$ holds for every $k=\alpha q+r$ with $\alpha <\alpha_0\in\N$.
Clearly, $N_{\mathcal L}^{\mathrm{red}}(\alpha_0 q+r,m) = N_{\mathcal L'}^{\mathrm{red}}(\alpha_0 q+r,m)=0$.
We proceed  by induction on $\zz$, decreasing from $m$ to $0$.
Suppose that $N_{\mathcal L}^{\mathrm{red}}(\alpha_0 q+r,\zz ) = N_{\mathcal L'}^{\mathrm{red}}(\alpha_0 q+r,\zz )$ for every $\zz>\zz_0$.
By \eqref{eq4:N^*-N^red}, $N_{\mathcal L}(\alpha_0 q+r,\zz_0)$ can be written as a linear combination  of the $N_{\mathcal L}^{\mathrm{red}}(\alpha q+r,\zz)$ for $\alpha\leq \alpha_0$ and $\zz \geq \zz_0$, and similarly for $N_{\mathcal L'}(\alpha_0 q+r,\zz_0)$.
Thus, by the inductive hypothesis, we obtain that $N_{\mathcal L}(\alpha_0 q+r,\zz_0) = N_{\mathcal L'}(\alpha_0 q+r,\zz_0)$ as asserted.
\end{proof}

\section{Computations and questions}\label{sec:examples}
In this section we shall use the finiteness theorem of Section~\ref{sec:finiteness} to produce, with the help of a computer, many examples of pairs of non-isometric congruence lattices that are $\norma{\cdot}^*$-isospectral.
In light of Theorem~\ref{thm3:characterization}~(ii), each such pair gives rise to a pair of non-isometric lens spaces that are $p$-isospectral for all $p$.

We next explain the computational procedure to find $\norma{\cdot}^*$-isospectral lattices.
For each $m$ and $q$, one finds first, by using Propositions~\ref{prop3:lens-isom} and \ref{prop3:isometrias}, a complete list of non-isometric $q$-congruence lattices in $\Z^m$.
Then, for each lattice $\mathcal L$ in the list, one computes the (finitely many) numbers $N_{\mathcal L}^{\mathrm{red}}(k,\zz )$ for $0\leq \zz \leq m$ and $0\leq k\leq (m-\zz)(q-1)$.
Next, for each pair of lattices, one compares their  associated  sets of numbers.
Finally, the program puts together the lattices for which these numbers coincide.
By Theorem~\ref{thm4:finitud}, such lattices are mutually $\norma{\cdot}^*$-isospectral.

By the procedure  above, using the computer program Sage~\cite{Sage}, we found all $\norma{\cdot}^*$-isospectral $m$-dimensional $q$-congruence lattices for $m=3$, $q\leq 300$  and $m=4$, $q\leq 150$ (see Tables~\ref{table:m=3} and \ref{table:m=4}).
We point out that all such lattices come in pairs  for these values of $q$ and $m$ (see Question~\ref{que6:families}).
In the tables, the parameters $[s_1,\dots,s_m]$ and $[s_1',\dots,s_m']$ in a row indicate the corresponding  $\norma{\cdot}^*$-isospectral lattices
$\mathcal L(q;s_1,\dots, s_m)$ and $\mathcal L(q;s_1',\dots, s_m')$ as in \eqref{eq3:Lambda(q;s)}.

\begin{table}
\caption{Pairs of $\norma{\cdot}^*$-isospectral $q$-congruence lattices of dimension $m=3$ for  $q\leq 300$.}\label{table:m=3}
{
\begin{tabular}[t]{c@{\;\;[}r@{,\,}r@{,\,}r@{\,]\quad[}r@{,\,}r@{,\,}r@{\,]\;\;}c}
$q$& $s_1$&$s_2$&$s_3$& $s_1'$&$s_2'$&$s_3'$\\ \hline
 49& 1&  6& 15& 1&  6& 20 & *\\
 64& 1&  7& 17& 1&  7& 23 & *\\
 98& 1& 13& 29& 1& 13& 41 & *\\
100& 1&  9& 21& 1&  9& 29 & *\\
100& 1&  9& 31& 1&  9& 39\\
121& 1& 10& 23& 1& 10& 32 & *\\
121& 1& 10& 34& 1& 10& 43\\
121& 1& 10& 45& 1& 10& 54\\
121& 1& 21& 34& 1& 21& 54\\
121& 1& 21& 45& 1& 21& 56\\
128& 1& 15& 33& 1& 15& 47 & *\\
147& 1& 20& 43& 1& 20& 62 & *\\
169& 1& 12& 27& 1& 12& 38 & *\\
169& 1& 12& 53& 1& 12& 64\\
169& 1& 12& 66& 1& 12& 77\\
169& 1& 25& 40& 1& 25& 64\\
169& 1& 25& 53& 1& 25& 77\\
169& 1& 38& 53& 1& 38& 79\\
169& 1& 12& 40& 1& 12& 51\\
169& 1& 25& 66& 1& 25& 79\\
192& 1& 23& 49& 1& 23& 71 & *\\
196& 1& 13& 29& 1& 13& 41 & *\\
196& 1& 13& 57& 1& 13& 69\\
196& 1& 41& 71& 1& 41& 85\\
196& 1& 13& 43& 1& 13& 55\\
196& 1& 13& 71& 1& 13& 83\\
196& 1& 27& 43& 1& 27& 69\\
196& 1& 27& 57& 1& 27& 83 & *\\
200& 1& 19& 41& 1& 19& 59 & *\\
200& 1& 19& 61& 1& 19& 79\\
242& 1& 21& 45& 1& 21& 65 & *
\end{tabular}
\qquad
\begin{tabular}[t]{c@{\;\;[}r@{,\,}r@{,\,}r@{\,]\quad[}r@{,\,}r@{,\,}r@{\,]\;\;}c}
$q$& $s_1$&$s_2$&$s_3$& $s_1'$&$s_2'$&$s_3'$\\ \hline
242& 1& 21& 67& 1& 21& 87 \\
242& 1& 21& 89& 1& 21&109 \\
242& 1& 43& 67& 1& 43&109 \\
242& 1& 43& 89& 1& 43&111 \\
245& 1& 34& 71& 1& 34&104 & *\\
256& 1& 15& 33& 1& 15& 47 & *\\
256& 1& 15& 81& 1& 15& 95 \\
256& 1& 31& 81& 1& 31&111 \\
256& 1& 47& 97& 1& 47&113 \\
256& 1& 15& 97& 1& 15&111 \\
256& 1& 31& 49& 1& 31& 79 \\
256& 1& 31& 65& 1& 31& 95 & *\\
289& 1& 16& 35& 1& 16& 50 & *\\
289& 1& 16& 86& 1& 16&101 \\
289& 1& 16&120& 1& 16&135 \\
289& 1& 33& 69& 1& 33&101 \\
289& 1& 33& 86& 1& 33&118 \\
289& 1& 50& 69& 1& 50&118 \\
289& 1& 50&103& 1& 50&137 \\
289& 1& 67& 86& 1& 67&137 \\
289& 1& 16& 52& 1& 16& 67 \\
289& 1& 16& 69& 1& 16& 84 \\
289& 1& 16&103& 1& 16&118 \\
289& 1& 33& 52& 1& 33& 84 \\
289& 1& 67&103& 1& 67&120 \\
289& 1& 33&103& 1& 33&135 \\
289& 1& 50& 86& 1& 50&135 \\
289& 1& 33&120& 1& 33&137 \\
294& 1& 41& 85& 1& 41&125 & *\\
300& 1& 29& 61& 1& 29& 89 & *\\
300& 1& 29& 91& 1& 29&119
\end{tabular}} \\[3mm]
Pairs marked with $*$  belong to the family to be given in Section~\ref{sec:families}.
\end{table}

\begin{table}
\caption{Pairs of $\norma{\cdot}^*$-isospectral $q$-congruence lattices of dimension $m=4$ for $q\leq 150$.}\label{table:m=4}
{
\begin{tabular}[t]{c@{\;\;[}r@{,\,}r@{,\,}r@{,\,}r@{\,]\quad[}r@{,\,}r@{,\,}r@{,\,}r@{\,]}}
$q$& $s_1$&$s_2$&$s_3$& $s_4$& $s_1'$&$s_2'$&$s_3'$&$s_4'$\\ \hline
 49& 1&  6&  8& 20& 1&  6&  8& 22 \\
 81& 1&  8& 10& 26& 1&  8& 10& 28 \\
 81& 1&  8& 10& 35& 1&  8& 10& 37 \\
 81& 1&  8& 19& 37& 1&  8& 26& 37 \\
 98& 1& 13& 15& 41& 1& 13& 15& 43 \\
100& 1&  9& 11& 29& 1&  9& 11& 31 \\
100& 1&  9& 21& 39& 1&  9& 29& 31 \\
121& 1& 10& 12& 32& 1& 10& 12& 34 \\
121& 1& 10& 12& 54& 1& 10& 12& 56 \\
121& 1& 10& 23& 56& 1& 10& 32& 56
\end{tabular}
\qquad
\begin{tabular}[t]{c@{\;\;[}r@{,\,}r@{,\,}r@{,\,}r@{\,]\quad[}r@{,\,}r@{,\,}r@{,\,}r@{\,]}}
$q$& $s_1$&$s_2$&$s_3$& $s_4$& $s_1'$&$s_2'$&$s_3'$&$s_4'$\\ \hline
121& 1& 10& 34& 54& 1& 10& 43& 45 \\
121& 1& 21& 23& 54& 1& 21& 23& 56 \\
121& 1& 10& 12& 43& 1& 10& 12& 45 \\
121& 1& 10& 23& 43& 1& 10& 32& 34 \\
121& 1& 10& 23& 45& 1& 10& 32& 54 \\
121& 1& 10& 23& 54& 1& 10& 32& 45 \\
121& 1& 10& 34& 56& 1& 10& 43& 56 \\
144& 1& 11& 13& 47& 1& 11& 13& 49 \\
144& 1& 11& 25& 59& 1& 11& 35& 49 \\
147& 1& 20& 22& 62& 1& 20& 22& 64
\end{tabular}}
\end{table}

Next we will attempt to explain in a unified manner the examples appearing in the tables.
Let $r$ and $t$ be positive integers and set $q=r^2t$, $r>1$.
We let $\ww= 1+rt$, considered as an element of $(\Z/q\Z)^\times$, the group of units of $\Z/q\Z$.
Then, the inverse of $\ww$ modulo $q$ is $\ww^{-1}:=1-rt$.
Clearly, for every $k \in \Z$,
\begin{equation*}
\ww^k \equiv 1+krt\pmod{q}.
\end{equation*}
In particular, $\ww$ has order $r$ in $(\Z/q\Z)^\times$.
For example, the pairs considered in Section~\ref{sec:families} can be written in the form
\begin{equation}\label{eq:r2t-inverses}
\mathcal L=\mathcal L(q;\ww^{0},\ww^{1}, \ww^{3})\qquad\text{and}\qquad
\mathcal L'=\mathcal L(q;\ww^{0},\ww^{-1}, \ww^{-3}).
\end{equation}

We note that all pairs in the tables have a description in terms of suitable powers of $\ww$ for some choices of $r$ and $t$ such that $q=r^2t$.
For instance, the simplest example in Table~\ref{table:m=3}, if we take $r=7$ and $t=1$ can be written as
\begin{equation}\label{eq6:ex-m=3-q=49}
\begin{array}{l}
\mathcal L(49;1,6,15)
    =
    \mathcal L(q;\ww^{0},-\ww^{-1}, \ww^{2})
    \cong_1
    \mathcal L(q;\ww^0, \ww^{1}, \ww^{3}),\\
\mathcal L(49;1,6,20)
    =
    \mathcal L(q;\ww^{0},-\ww^{-1}, -\ww^{-3})
    \cong_1
    \mathcal L(q;\ww^{0},\ww^{-1}, \ww^{-3}),
\end{array}
\end{equation}
where $\cong_1$ denotes isometric in $\norma{\cdot}$.
Indeed, in both cases we multiplied by an appropriate power of $\theta$ and then we reordered the terms.
Furthermore, the first pair in Table~\ref{table:m=4}, if $r=7$ and $t=1$ becomes
\begin{equation}\label{eq6:ex-m=4-q=49}
\begin{array}{l}
\mathcal L(49;1,6,8,20)
    =
    \mathcal L(q;\ww^0, -\ww^{-1}, \ww^{1}, -\ww^{-3})
    \cong_1
    \mathcal L(q;\ww^0, \ww^{2}, \ww^{3}, \ww^{4}),\\
\mathcal L(49;1,6,8,22)
    =
    \mathcal L(q;\ww^0, -\ww^{-1}, \ww^{1}, \ww^{3})
    \cong_1
    \mathcal L(q;\ww^0, \ww^{-2}, \ww^{-3}, \ww^{-4}).
\end{array}
\end{equation}

We point out that all examples shown in Tables~\ref{table:m=3} and \ref{table:m=4} respond to the following description:
\begin{equation}\label{eq6:inversos}
\mathcal L(q;\ww^{d_0},\ww^{d_1},\dots,\ww^{d_{m-1}})
\quad\text{and}\quad
\mathcal L(q;\ww^{-d_0},\ww^{-d_1},\dots,\ww^{-d_{m-1}}),
\end{equation}
where $q=r^2t$, $r>1$,  $\ww=1+rt$ and $0=d_0<d_1<\dots<d_{m-1}<r$.
However, note that for some choices of $m$, $r$ and $t$, there are sequences $0=d_0<d_1<\dots<d_{m-1}<r$ such that the lattices defined  as in \eqref{eq6:inversos} are not $\norma{\cdot}^*$-isospectral.
For example, this is the case when
$m=3$, $r=8$, $t=1$ and
$[d_0,d_1,d_2]=[0,1,4]$.

The following questions come up naturally.
\begin{question}\label{que6:conditions}
Give conditions on the sequence $0=d_0<d_1<\dots<d_{m-1}<r$ for   lattices as in \eqref{eq6:inversos} to be  $\norma{\cdot}^*$-isospectral.
\end{question}
\begin{question}
Are there examples of $\norma{\cdot}^*$-isospectral lattices that are not of the type in \eqref{eq6:inversos} for some choice of $\theta$?
\end{question}
\begin{question}\label{que6:families}
Are there families of  $\norma{\cdot}^*$-isospectral lattices having more than two elements?
\end{question}
We have carried out computations for small values of  $m$ and $q$ and in this search
we have not found any such family  yet.

Next, we give a particular sequence as in \eqref{eq6:inversos} that is very likely to give $\norma{\cdot}^*$-isospectral pairs in all dimensions under rather general conditions on $r$, for instance, if $r$ is prime. This motivation makes it worth showing that this sequence always gives non-isometric lattices.

\begin{prop}\label{lem6:non-isom}
Let $m\geq3$, $r\geq m+3$, $t\in\N$ and set $q=r^2t$ and $\ww=1+rt$, then the $q$-congruence lattices
\begin{equation}\label{eq6:inv-non-isom}
\mathcal L(q;\ww^0, \ww^2, \ww^{3}, \dots, \ww^{m})
\quad\text{and}\quad
\mathcal L(q;\ww^0, \ww^{-2}, \ww^{-3}, \dots, \ww^{-m})
\end{equation}
are not $\norma{\cdot}$-isometric.
\end{prop}
\begin{proof}
In general, if $0=d_0<d_1<\dots<d_{m-1}<r$, we associate to $\mathcal L(q;\ww^{d_0},\ww^{d_1},\dots,\ww^{d_{m-1}})$ the following ordered partition of $r$:
$$
r=(d_1-d_0)+\dots+(d_{m-1}-d_{m-2})+(r-d_{m-1}).
$$
By using Propositions~\ref{prop3:lens-isom} and \ref{prop3:isometrias}, one can check that two lens spaces with partitions $r=a_1+\dots+a_m$ and $r=b_1+\dots+b_m$ are $\norma{\cdot}$-isometric if and only if there is $l \in \Z$ such that $a_j=b_{j+l}$ for every $j$, where the index $j+l$ is taken in the interval  $[0,m-1]$$ \pmod m$.

In our case, the ordered partitions for the lattices $\mathcal L(q;\ww^0, \ww^2, \ww^{3}, \dots, \ww^{m})$ and $\mathcal L(q;\ww^0, \ww^{-2}, \ww^{-3}, \dots, \ww^{-m}) {\cong}_1 \mathcal L(q;\ww^0, \ww^{1}, \dots, \ww^{m-2},\ww^{m})$ are
\begin{align*}
r=2+1+\dots+1+(r-m),\\
r=1+\dots+1+2+(r-m).
\end{align*}
The assertion now follows since $r-m\geq3$.
\end{proof}

\section{Families of $\norma{\cdot}^*$-isospectral lattices} \label{sec:families}

The goal of this section is to construct an infinite two-parameter family of pairs of  $\norma{\cdot}^*$-isospectral lattices in $\Z^m$ for $m=3$.
Together with Theorem~\ref{thm3:characterization}~(ii), this will produce infinitely many pairs of non-isometric $5$-dimensional lens spaces, isospectral on $p$-forms for every $p$.
Although our construction does not give all of the examples for $m=3$, the list given in Section~\ref{sec:examples} shows that most of the examples can be obtained by a slight variation of the method used in this section.

Throughout this section, we fix $r,t\in\N$, $r>1$. We single out (see \eqref{eq3:Lambda(q;s)}) the congruence lattices
\begin{equation}
\begin{array}{rcl}
\mathcal L  &=& \mathcal L(\;r^2t;\;1,\;1+rt,\;1+3rt),\\
\mathcal L' &=& \mathcal L(\;r^2t;\;1,\;1-rt,\;1-3rt).
\end{array}
\end{equation}
In other words, $\mathcal L$ and $\mathcal L'$ are defined by the equations
\begin{equation}\label{eq5:L-L'}
\begin{array}{rc@{\;}c@{\;}c@{\;}c@{\;}c@{\;\;}c@{\;\;}l}
\mathcal L:&\quad
    a&+&(1+rt)b&+&(1+3rt)c&\equiv &0 \pmod{r^2t},\\[1mm]
\mathcal L':&\quad
a&+&(1-rt)b&+&(1-3rt)c&\equiv &0 \pmod{r^2t},
\end{array}
\end{equation}
or equivalently by
\begin{equation}\label{eq5:L-L'2}
\begin{array}{rc@{\;}c@{\;}c@{\;}c@{\;}c@{\;\;}c@{\;\;}l}
\mathcal L:&\quad
    a+b+c&+rt(b+3c)&\equiv &0 \pmod{r^2t},\\[1mm]
\mathcal L':&\quad
    a+b+c&-rt(b+3c)&\equiv &0 \pmod{r^2t}.
\end{array}
\end{equation}

Our goal is to prove that $\mathcal L$ and $\mathcal L'$ are $\norma{\cdot}^*$-isospectral for every $r$ not divisible by $3$.
By Proposition~\ref{lem6:non-isom}, $\mathcal L$ and $\mathcal L'$ are non-isometric for $r\geq6$.
The first pair in this family is for $r=7$ and $t=1$, namely $\mathcal L=\mathcal L(49;1,8,22)$ and $\mathcal L'=\mathcal L(49;1,-6,-20)$.
We point out that this pair is isometric to the simplest pair in Table~\ref{table:m=3} by \eqref{eq6:ex-m=3-q=49}.

We recall that $\mathcal L$ and $\mathcal L'$ are said to be $\norma{\cdot}^*$-isospectral if $N_{\mathcal L}(k,\zz ) = N_{\mathcal L'}(k,\zz )$ for every $k\in\N$ and every $0\leq \zz \leq m=3$ (where $N_{\mathcal L}(k,\zz )$ is as in \eqref{eq:Nkz}).
We shall first prove that this equality holds easily for $\zz =1,2,3$.

\begin{lemma}\label{lem:z=1,2,3}
For any $\zz =1,2,3$ and any $k \in \N$, one has that $N_\mathcal L (k,\zz )= N_{\mathcal L'} (k,\zz )$.
\end{lemma}

\begin{proof}
The assertion is clear for $\zz =3$. Also,  it is easy to check for $\zz =2$ since the elements $(sr^2t,0,0),(0,sr^2t,0),(0,0,sr^2t)$ for $s\in\Z$, $s\neq0$ are the only ones in both lattices having exactly two coordinates equal to zero.

For $\zz =1$, it is not hard to give a $\norma{\cdot}$-preserving bijection between the sets $\{\eta\in\mathcal L:Z(\eta)=1\}$ and $\{\eta'\in\mathcal L':Z(\eta')=1\}$. Namely one has
\begin{align*}
(a,b,0)\in\mathcal L\quad&\Longleftrightarrow\quad (b,a,0)\in\mathcal L', \\
(a,0,c)\in\mathcal L\quad&\Longleftrightarrow\quad (c,0,a)\in\mathcal L', \\
(0,b,c)\in\mathcal L\quad&\Longleftrightarrow\quad (0,c,b)\in\mathcal L',
\end{align*}
for every nonzero integers $a$, $b$ and $c$.
For example, $(a,b,0)\in\mathcal L$ $\iff$ $a+(1+rt)b\equiv0\pmod{r^2t}$ $\iff$ $(1-rt)a+b\equiv0\pmod{r^2t}$ $\iff$ $(b,a,0)\in\mathcal L$.
The second and the third rows follow in a similar way, multiplying by $1-3rt$ and $1-4rt$ respectively.
\end{proof}

\begin{rem}
It now remains to prove that $N_\mathcal L (k,0) = N_{\mathcal L'}(k,0)$ for every $k$, which,
by Lemma~\ref{lem:z=1,2,3}, is equivalent to show that $\mathcal L$ and $\mathcal L'$ are $\norma{\cdot}$-isospectral, since $N_{\mathcal L} (k) = \sum _{\zz =0}^3 N_\mathcal L (k,\zz )$.
We see that, remarkably, in light of Theorem~\ref{thm3:characterization}~(i), the previous lemma allows us to reduce the verification of $p$-isospectrality for all $p$ of the associated lens spaces, to prove that they are just $0$-isospectral.
\end{rem}

\begin{thm}\label{thm4:isolattices}
For any $r$ and $t$ positive integers with $r\not\equiv0 \pmod3$, the lattices $\mathcal L$ and $\mathcal L'$ in \eqref{eq5:L-L'} are $\norma{\cdot}^*$-isospectral.
\end{thm}

\begin{proof}
By Lemma~\ref{lem:z=1,2,3}, it remains to prove that $N_\mathcal L (k,0) = N_{\mathcal L'}(k,0)$ for every $k$.
This is clearly true for $k=0$, hence we will assume that $k>0$.
The proof consists in showing that the number of elements with a fixed one-norm in each octant is the same for both lattices.
Since lattices have central symmetry, we have
\begin{align*}
  \tfrac12 \,N_{\mathcal L}(k,0) = N_{\mathcal L}^{{+}{+}{+}}(k) + N_{\mathcal L}^{{+}{+}{-}}(k) + N_{\mathcal L}^{{+}{-}{+}}(k) + N_{\mathcal L}^{{+}{-}{-}}(k),
\end{align*}
where the signs in the supra-indexes indicate the signs of the coordinates.
That is,  $N_{\mathcal L}^{{+}{+}{-}}(k)$ is the number of $\eta=(a,b,c)\in\mathcal L$ such that $\norma{\eta}=k$, $a>0$, $b>0$ and $c<0$.
We will show that $N_{\mathcal L}^{{+}{+}{+}}(k)=N_{\mathcal L'}^{{+}{+}{+}}(k)$, $N_{\mathcal L}^{{+}{+}{-}}(k)=N_{\mathcal L'}^{{+}{+}{-}}(k)$ and so on.

We first examine the octant ${+}{+}{+}$.
Any vector here has the form
\begin{equation}\label{eq5:eta+++}
\eta=(k-x,x-y,y)
\qquad\text{with} \quad 0<y<x<k.
\end{equation}
By \eqref{eq5:L-L'2}, we have that $\eta\in\mathcal L$ (resp.\ $\eta\in\mathcal L'$) if and only if $k+rt(x+2y)\equiv 0\pmod{r^2t}$ (resp.\ $k-rt(x+2y)\equiv 0\pmod{r^2t}$).
Thus $N_{\mathcal L}^{{+}{+}{+}}(k)=N_{\mathcal L'}^{{+}{+}{+}}(k)=0$ unless $k$ is divisible by $rt$.
We write $k=\omega rt$ for some positive integer $\omega$.
Then
\begin{equation}\label{eq5:+++}
\begin{array}{lcr@{\;}l}
\eta\in\mathcal L&
\Longleftrightarrow &
&x+2y\equiv -\omega\pmod{r},
\\[1mm]
\eta\in\mathcal L'&
\Longleftrightarrow &
&x+2y\equiv \;\;\,\omega\pmod{r}.
\end{array}
\end{equation}

In order to count the number of solutions in \eqref{eq5:+++}, we split the set of integer points $(x,y)$ satisfying $0<y<x<k$ into squares and triangles as follows.
We take the squares $\{(x,y): \alpha r\leq x<(\alpha+1)r, \; \beta r<y\leq (\beta+1)r\}$ for $1\leq\beta<\alpha\leq \omega t-1$, together with the triangles $\{(x,y): \alpha r<x<(\alpha+1),\; \alpha r<y<x\}$ for $0\leq\alpha\leq wt-1$.
We note that the points which are the upper-left corner of the squares near the diagonal (when $\alpha=\beta+1$) are not in the orginal set; this will be taken into account in the computations.
There are $\binom{\omega t}{2}$ squares and $\omega t$ triangles.
Set
\begin{equation*}
A(r,\xi)=\#\{(x,y)\in\Z^2: 0<y<x<r, \; x+2y\equiv\xi\pmod r\}.
\end{equation*}
Since we are working modulo $r$, the number of elements of $\mathcal L$ (resp.\ $\mathcal L'$) in any triangle is always the same and is equal to  $A(r,-\omega)$ (resp.\ $A(r,\omega)$).
Thus we have $\omega t A(r,-\omega)$ (resp.\ $\omega t A(r,\omega)$) elements in $\mathcal L$ (resp.\ $\mathcal L'$) in the union of all triangles.
On the other hand, if $\omega\not\equiv 0\pmod r$, there are exactly $r$ elements in $\mathcal L$ (or in $\mathcal L'$) in each square, thus we have $\binom{\omega t}{2} r$ elements in $\mathcal L$ (or in $\mathcal L'$) and in the union of all the squares.
When $\omega\equiv 0\pmod r$, one has the same quantity minus $\omega t-1$ elements, since, as noticed above, we have to exclude the vertices $(\alpha \omega t,(\alpha+1) \omega t)$ for $1\leq \alpha\leq \omega t-1$ which lie
in the squares next to the diagonal $x=y$.
Summing up, we get
\begin{equation*}
N_{\mathcal L}^{{+}{+}{+}}(\omega rt) =
\begin{cases}
\omega t\, A(r,-\omega) + \binom{\omega t}{2} r
  &\quad\text{if $\omega\not\equiv 0\pmod r$},\\
\omega t\, A(r,-\omega) + \binom{\omega t}{2} r-\omega t +1
  &\quad\text{if $\omega\equiv 0\pmod r$},
\end{cases}
\end{equation*}
and the same for $N_{\mathcal L'}^{{+}{+}{+}}(\omega rt)$ replacing $A(r,-\omega)$ by $A(r,\omega)$.
The next lemma gives a formula for $A(r,\xi)$ showing that $A(r,\omega)=A(r,-\omega)$, hence $N_{\mathcal L}^{{+}{+}{+}}(\omega rt)=N_{\mathcal L'}^{{+}{+}{+}}(\omega rt)$.

\begin{lemma}\label{lem5:A(r,omega)}
Let $r$ and $\xi$ be integers such that $r\not\equiv0\pmod3$.
If $r$ is odd, then
\begin{equation*}
A(r,\xi) =
\begin{cases}
\tfrac{r-3}{2}
\quad&\text{if $\xi\not\equiv0\pmod r$,} \\
\tfrac{r-1}{2}
\quad&\text{if $\xi\equiv0\pmod r$.}
\end{cases}
\end{equation*}
If $r$ is even, then
\begin{equation*}
A(r,\xi) =
\begin{cases}
\tfrac{r}2-1
\quad&\text{if $\xi\not\equiv0\pmod r$ and $\xi$ is odd,}\\
\tfrac{r}2-2
\quad&\text{if $\xi\not\equiv0\pmod r$ and $\xi$ is even,}\\
\tfrac{r}2-1
\quad&\text{if $\xi\equiv0\pmod r$.}
\end{cases}
\end{equation*}
\end{lemma}
We will often use the standard notation $\lfloor u\rfloor=\max\{d\in\Z: d\leq u\}$ and $\lceil u\rceil=\min\{d\in\Z: d\geq u\}$ for the floor and ceiling of a real number respectively.
\begin{proof}
We may assume that $0\leq \xi<r$.
Suppose that $x+2y\equiv \xi\pmod r$; thus, $x=\gamma r+\xi-2y$ for some $\gamma\in\Z$.
One can check that $1\leq\gamma\leq 2$ if $\xi=0$ and $0\leq \gamma\leq 2$ if $\xi>0$, since $0<y<x<r$.
Furthermore, the restrictions $y<x$ and $x<r$ are equivalent to
\begin{equation}\label{eq5:intervalo-beta}
\begin{array}{rcl@{\hspace{10mm}}rcl}
y+1&\leq& \gamma r+\xi-2y,  &
    \gamma r+\xi-2y&\leq& r-1\\[1mm]
y&\leq &\tfrac{\gamma r+\xi-1}{3}, &
    \tfrac{(\gamma-1) r+\xi+1}{2}&\leq& y\\[1mm]
y&\leq &\lfloor\tfrac{\gamma r+\xi-1}{3}\rfloor, &
    \lceil\tfrac{(\gamma-1) r+\xi+1}{2}\rceil&\leq& y.
\end{array}
\end{equation}
If $\xi=0$, then $\gamma=1$ implies $1\leq y\leq \lfloor\tfrac{r-1}{3}\rfloor$ and $\gamma=2$ implies $\lceil\tfrac{r+1}{2}\rceil\leq y\leq \lfloor\tfrac{2r-1}{3}\rfloor$, thus
$$
A(r,0) = \lfloor\tfrac{r-1}{3}\rfloor + \lfloor\tfrac{2r-1}{3}\rfloor+1-\lceil\tfrac{r+1}{2}\rceil = r-\lceil\tfrac{r+1}{2}\rceil,
$$
which is our assertion for $\xi\equiv0\pmod r$.

Similarly, if $\xi>0$, then $\gamma=0$ implies $1\leq y\leq \lfloor\tfrac{\xi-1}{3}\rfloor$, $\gamma=1$ implies $\lceil\tfrac{\xi+1}{2}\rceil\leq y\leq \lfloor\tfrac{r+\xi-1}{3}\rfloor$ and $\gamma=2$ implies $\lceil\tfrac{r+\xi+1}{2}\rceil\leq y\leq \lfloor\tfrac{2r+\xi-1}{3}\rfloor$, thus
$$
A(r,\xi) = \lfloor\tfrac{\xi-1}{3}\rfloor+ \lfloor\tfrac{r+\xi-1}{3}\rfloor + \lfloor\tfrac{2r+\xi-1}{3}\rfloor
+2-\lceil\tfrac{\xi+1}{2}\rceil -\lceil\tfrac{r+\xi+1}{2}\rceil
=r+\xi-\big( \lceil\tfrac{\xi+1}{2}\rceil +\lceil\tfrac{r+\xi+1}{2}\rceil\big).
$$
The rest of the proof is straightforward.
\end{proof}

We continue with the proof of Theorem~\ref{thm4:isolattices}, now considering the octant ${+}{-}{-}$.
Any vector in this octant can be written as
\begin{equation}\label{eq5:eta+--}
\eta=(k-x,y-x,-y)
\qquad\text{with} \quad 0<y<x<k,
\end{equation}
then, by \eqref{eq5:L-L'2}, we have that
\begin{equation}\label{eq5:+--1}
\begin{array}{lcr@{\;}l}
\eta\in\mathcal L&
\Longleftrightarrow &
&\;\;\,rt(x+2y)\equiv k-2x\pmod{r^2t},
\\[1mm]
\eta\in\mathcal L'&
\Longleftrightarrow &
&-rt(x+2y)\equiv k-2x\pmod{r^2t}.
\end{array}
\end{equation}
In both cases we have $2x\equiv k\pmod{rt}$.
We fix $k=\omega rt+k_0$ with $0\leq k_0<rt$, thus $x$ must satisfy $2x\equiv k_0\pmod{rt}$.

We first assume that $rt$ is odd.
Then there exists only one $x_0$ satisfying $2x_0\equiv k_0\pmod{rt}$ and $0\leq x_0<rt$.
We write any other solution as $x=\alpha rt+x_0$ for some $\alpha$.
The restriction $0<x<k$ is equivalent to
\begin{equation}\label{eq5:intervalo-alpha}
\begin{array}{rcccl}
1&\leq& \alpha rt+x_0 &\leq &\omega rt+k_0-1, \\[1mm]
-\lfloor\tfrac{x_0-1}{rt}\rfloor=\lceil\tfrac{1-x_0}{rt}\rceil &\leq& \alpha &\leq& \omega+ \lfloor\tfrac{k_0-1-x_0}{rt}\rfloor.
\end{array}
\end{equation}
On the other hand, by \eqref{eq5:+--1}, we have that
\begin{equation}\label{eq5:+--2}
\begin{array}{lcr@{\;}l}
\eta\in\mathcal L&
\Longleftrightarrow &
2y\equiv &-x_0+\left(\omega-2\alpha + \tfrac{k_0-2x_0}{rt}\right) \pmod{r},
\\[1mm]
\eta\in\mathcal L'&
\Longleftrightarrow &
2y\equiv &-x_0-\left(\omega-2\alpha + \tfrac{k_0-2x_0}{rt}\right) \pmod{r}.
\end{array}
\end{equation}
Since $r$ is odd, these equations always have a solution $y$, which is unique modulo $r$.
For each $\alpha$ satisfying \eqref{eq5:intervalo-alpha}, denote respectively by $y_\alpha$ and $y_\alpha'$ the solutions of \eqref{eq5:+--2} for $\mathcal L$ and $\mathcal L'$ such that $0\leq y_\alpha,y_\alpha'<r$.
We write the solutions as $y=\beta r+y_\alpha$ and $y'=\beta' r+y_{\alpha'}$.
Now, the restriction $0<y<x$ is equivalent to
\begin{equation}\label{eq5:intervalo-beta}
\begin{array}{rcccl}
1&\leq& \beta r+y_\alpha &\leq &\alpha rt+x_0-1, \\[1mm]
-\lfloor\tfrac{y_\alpha-1}{r}\rfloor=\lceil\tfrac{1-y_\alpha}{r}\rceil &\leq& \beta &\leq& \alpha t+ \lfloor\tfrac{x_0-1-y_\alpha}{r}\rfloor.
\end{array}
\end{equation}
Hence
\begin{equation}\label{eq5:N_L}
N_{\mathcal L}^{{+}{-}{-}}(k)
= \sum_{\alpha=-\lfloor\tfrac{x_0-1}{rt}\rfloor}^{\omega+ \lfloor\tfrac{k_0-1-x_0}{rt}\rfloor}
    \left(\alpha t+1+ \lfloor \tfrac{x_0-1-y_\alpha}{r} \rfloor+\lfloor \tfrac{y_\alpha-1}r\rfloor\right).
\end{equation}
The same formula holds for $N_{\mathcal L'}^{{+}{-}{-}}(k)$ replacing $y_\alpha$ by $y_\alpha'$.
Then $N_{\mathcal L}^{{+}{-}{-}}(k)-N_{\mathcal L'}^{{+}{-}{-}}(k)$ is equal to
\begin{equation}\label{eq5:N_L-N_L'}
C:=\sum_{\alpha=-\lfloor\tfrac{x_0-1}{rt}\rfloor}^{\omega+ \lfloor\tfrac{k_0-1-x_0}{rt}\rfloor}
    \left(\lfloor \tfrac{x_0-1-y_\alpha}{r} \rfloor+\lfloor \tfrac{y_\alpha-1}r\rfloor-\lfloor \tfrac{x_0-1-y_\alpha'}{r} \rfloor-\lfloor \tfrac{y_\alpha'-1}r\rfloor\right)
\end{equation}

The proof in the case when $rt$ is odd will be completed by showing that $C=0$.
We first suppose that $k_0$ is even and nonzero, thus $k_0=2x_0$ with $0<x_0<rt/2$ and $x_0<k_0$.
Then, \eqref{eq5:N_L-N_L'} implies that
\begin{equation*}
C=\sum_{\alpha=0}^{\omega}
    \left(\lfloor \tfrac{x_0-1-y_\alpha}{r} \rfloor
    -\lfloor \tfrac{x_0-1-y_\alpha'}{r} \rfloor
    +\lfloor \tfrac{y_\alpha-1}r\rfloor
    -\lfloor \tfrac{y_\alpha'-1}r\rfloor\right).
\end{equation*}
But a careful look at \eqref{eq5:+--2} shows that the solutions of both equations are related by the equation $y_\alpha'=y_{\omega-\alpha}$ for every $0\leq \alpha\leq\omega$; hence, $C=0$.

If $k_0=0$, then $x_0=0$ and the sum in \eqref{eq5:N_L-N_L'} runs through the interval $1\leq\alpha\leq \omega-1$.
Hence, $C=0$ since $y_\alpha'=y_{\omega-\alpha}$ for every $1\leq \alpha\leq\omega-1$ by \eqref{eq5:+--2}.

Now suppose that $k_0$ is odd, then $k_0=2x_0-rt$ with $rt/2\leq x_0<rt$ and $k_0<x_0$.
In this case the sum in \eqref{eq5:N_L-N_L'} runs through the interval $0\leq\alpha\leq\omega-1$ and $y_\alpha'=y_{\omega-1-\alpha}$ for every $1\leq\alpha\leq \omega-1$ by \eqref{eq5:+--2}, hence $C=0$.

We now assume that $rt$ is even.
We recall that $x$ must satisfy $2x\equiv k_0\pmod{rt}$.
Clearly, when $k$ is odd, $N_{\mathcal L}^{{+}{-}{-}}(k) = N_{\mathcal L'}^{{+}{-}{-}}(k)=0$; thus, we assume that $k$ is even.
Let $x_0$ be the only integer such that $2x_0\equiv k_0\pmod{rt}$ and $0\leq x_0<\frac{rt}{2}$.
Thus $x_0=k_0/2\leq k_0$ and a general solution has the form $x=\alpha\frac{rt}{2}+x_0$.
Similarly, as in \eqref{eq5:intervalo-alpha}, one can check that the restriction $0<x<k$ is equivalent to $-\lfloor2\tfrac{x_0-1}{rt}\rfloor \leq \alpha\leq 2\omega + \lfloor2\tfrac{k_0-x_0-1}{rt}\rfloor$, or more precisely, $1\leq \alpha\leq 2\omega-1$ if $x_0=0$ and, $0\leq \alpha \leq 2\omega$ if $x_0>0$.
Thus in this case, from \eqref{eq5:+--1} we have that
\begin{equation}\label{eq5:+--3}
\begin{array}{lcr@{\;}l}
\eta\in\mathcal L&
\Longleftrightarrow &
2y\equiv &-x_0-\alpha\tfrac{rt}{2}+(\omega-\alpha) \pmod{r},
\\[1mm]
\eta\in\mathcal L'&
\Longleftrightarrow &
2y\equiv &-x_0-\alpha\tfrac{rt}{2}-(\omega-\alpha) \pmod{r}.
\end{array}
\end{equation}
If $r$ is odd, then both equations always have a solution $y$, which is unique modulo $r$.
When $r$ is even, we assume that $-x_0-\alpha\tfrac{rt}{2}+\omega-\alpha$ is even since both equations do not have any solution otherwise.
Thus, equations in \eqref{eq5:+--3} have unique solutions modulo $\tfrac r2$.
Let $y_\alpha$ and $y_\alpha'$ be the smallest non-negative solutions of \eqref{eq5:+--3} for $\mathcal L$ and $\mathcal L'$ respectively.
A similar argument as in \eqref{eq5:intervalo-beta} implies that $N_{\mathcal L}^{{+}{-}{-}}(k)$ is equal to the sum over $-\lfloor2\tfrac{x_0-1}{rt}\rfloor \leq \alpha\leq 2\omega + \lfloor2\tfrac{k_0-x_0-1}{rt}\rfloor$ of the terms
\begin{equation}
\begin{cases}
\alpha\tfrac{t}{2} +1 +\lfloor\tfrac{x_0-y_\alpha-1}{r}\rfloor+ \lfloor\tfrac{y_\alpha-1}{r}\rfloor
    &\text{ if $r$ is odd},\\
\alpha t +1 +\lfloor2\tfrac{x_0-y_\alpha-1}{r}\rfloor+ \lfloor2\tfrac{y_\alpha-1}{r}\rfloor
    &\text{ if $r$ is even},
\end{cases}
\end{equation}
and the same formula holds for $N_{\mathcal L'}^{{+}{-}{-}}(k)$ replacing $y_\alpha$ by $y_\alpha'$.
But, for arbitrary $r$, \eqref{eq5:+--3} implies that $y_\alpha' = y_{2\omega-\alpha}$ for every $0\leq \alpha\leq 2\omega$; then, $N_{\mathcal L}^{{+}{-}{-}}(k)=N_{\mathcal L'}^{{+}{-}{-}}(k)$.
This concludes the proof for the octant ${+}{-}{-}$.

Entirely similar arguments apply to the octants ${+}{+}{-}$ and ${+}{-}{+}$, by considering the elements written as $(k-x,x-y,-y)$ and $(k-x,-y,x-y)$ for $0<y<x<k$ respectively.
\end{proof}

\begin{rem}
The previous proof gives an explicit formula for $N_{\mathcal L}^{{+}{+}{+}}(k)$, $N_{\mathcal L}^{{+}{-}{-}}(k)$, $N_{\mathcal L}^{{+}{+}{-}}(k)$ and $N_{\mathcal L}^{{+}{-}{+}}(k)$ for every $k$; thus, also for $N_{\mathcal L}(k,0)$.
Actually, we have checked with the computer that the formulas hold for $k\leq 1000$.
A formula for $N_{\mathcal L}^{{+}{+}{+}}(k)$ was included before Lemma~\ref{lem5:A(r,omega)}.
An explicit expression for $N_{\mathcal L}^{{+}{-}{-}}(k)$ could also be given but the formula must be divided into many cases, namely, $rt$ odd, $rt$ even and $r$ odd, $rt$ even and $r$ odd, and (following the notation inside the proof) with each of these subdivided into $k_0$ odd, $k_0>0$ even, $k_0=0$ (subdivided again by $y_\alpha=0$, $y_\alpha>0$).
Similar complications occur for the octants ${+}{+}{-}$ and ${+}{-}{+}$.

Any of these expressions mentioned above contains a main term and a residual term  written as a sum of floors of rational numbers.
For example, when $rt$ is odd and $k$ is even and not divisible by $rt$, \eqref{eq5:+--1} implies that
\begin{equation*}
N_{\mathcal L}^{{+}{-}{-}}(k) = t\binom{\omega+1}{2}+\omega+1+\sum_{\alpha=0}^{\omega}
    \left(\lfloor\tfrac{x_0-1-y_\alpha}{r}\rfloor+\lfloor\tfrac{y_\alpha-1}{r}\rfloor\right),
\end{equation*}
where $\omega=\lfloor k/{rt}\rfloor$, $x_0$ is the only integer such that $2x_0\equiv k\pmod {rt}$ and $0 \leq x_0<rt$ and $y_\alpha$ is the only solution of $2y_\alpha\equiv -x_0-\omega+2\alpha\pmod r$ satisfying $0\leq y_\alpha<r$.

It is easy to give an expression for $N_{\mathcal L}(k,\zz)$ for $\zz$ equal to $2$ and $3$.
It is also possible for $\zz=1$ in a similar ---and simpler--- way as in the previous proof.
This implies that we can compute explicitly every $p$-spectrum of the lens space $L(r^2t; 1, 1+rt, 1+3rt)$ by using the formula in Theorem~\ref{thm3:dim V_k,p^Gamma}.
\end{rem}

\begin{rem}
In a previous version \cite{LMRhodgeiso_old} of this article, we proved Theorem~\ref{thm4:isolattices} (for $rt$ odd) with a completely different method which was more involved but gave useful additional geometric information on the lattices.
\end{rem}

\section{Lens spaces $p$-isospectral for every $p$}\label{sec:all-p-iso}

In this section we summarize the spectral properties of lens spaces that can be obtained from the results on congruence lattices in the previous three sections, in light of the characterization in Theorem~\ref{thm3:characterization}.
It also contains information  on the geometric, topological and spectral properties of the examples.

\begin{thm}\label{thm7:all-p-iso-lens}
For any $r$ and $t$ positive integers with $r\geq7$ and $r\not\equiv0 \pmod3$, the lens spaces
\begin{equation*}
L(r^2t; 1,1+rt,1+3rt)
\quad\text{and}\quad
L(r^2t; 1,1-rt,1-3rt)
\end{equation*}
are $p$-isospectral for all $p$ but not strongly isospectral.
\end{thm}

Tables~\ref{table:m=3} and \ref{table:m=4} give more such pairs in dimensions $5$ and $7$ respectively.
The proof of $p$-isospectrality for all $p$ follows immediately from Theorems~\ref{thm3:characterization} and \ref{thm4:isolattices}.
The non-isometry comes from Proposition~\ref{lem6:non-isom}.
They are not strongly isospectral by the following general fact, which follows from well known results.
We include a proof for completeness.

\begin{prop}\label{prop7:lens-non-strongly}
If $L$ and $L'$ are strongly isospectral lens spaces, then they are isometric.
\end{prop}

\begin{proof}
We first assume  that $\Gamma\ba S^{2m-1}$ and $\Gamma'\ba S^{2m-1}$ are strongly isospectral spherical spaces forms, where $\Gamma$ and $\Gamma'$ are arbitrary finite subgroups of $\Ot(2m)$ acting freely on $S^{2m-1}$.
By Proposition~1 in \cite{Pe1}, the subgroups $\Gamma$ and $\Gamma'$ are representation equivalent, i.e.\ $L^2(\Gamma\ba \Ot(2m))$ and $L^2(\Gamma'\ba \Ot(2m))$ are equivalent representations of $\Ot(2m)$.
Hence, $\Gamma$ and $\Gamma'$ are almost conjugate in $\Ot(2m)$ (see Lemma~2.12 in \cite{Wo2}).

In our case, $L=\Gamma\ba S^{2m-1}$ and $L'=\Gamma'\ba S^{2m-1}$ are lens spaces with $\Gamma$ and $\Gamma'$ cyclic subgroups of $\SO(2m)$.
Since almost conjugate cyclic subgroups are necessarily conjugate, then $L$ and $L'$ are isometric.
\end{proof}

We observe that the examples in Theorem~\ref{thm7:all-p-iso-lens} allow to obtain pairs of Riemannian manifolds in every dimension $n\ge 5$ that are $p$-isospectral for all $p$ and are not strongly isospectral.
Indeed, for this purpose, we may just take $M=L\times S^k$ and $M'=L'\times S^k$, for any $k\in \N_0$, where $L$, $L'$ is any pair of non-isometric lens spaces in dimension $5$ satisfying $p$-isospectrality for every $p$.
In relation to lens spaces of higher dimensions we have the following result.

\begin{thm}\label{thm7:high-dim}
For any $n_0 \ge 5$, there are pairs of non-isometric  lens spaces of dimension $n$, with $n>n_0$,  which are $p$-isospectral for all $p$.
\end{thm}

\begin{proof}
We will apply Theorem~\ref{thm4:isolattices}, together with an extension of a duality result of Ikeda. For each $q\in\N$ and $n=2m-1$ odd, denote by $\mathfrak L_0(q,m)$ the classes of non-isometric $n$-dimensional lens spaces $L(q;s_1,\dots,s_m)$ such that $s_i\not\equiv \pm s_j\pmod q$ for all $i\neq j$.
Set $h=(\phi(q)-2m)/2$, where $\phi$ is the Euler function.

To each lens space $L=L(q;s_1,\dots,s_m)$ in $\mathfrak L_0(q,m)$, one associates the lens space $\overline L=L(q;\bar s_1,\dots,\bar s_h)$, where the parameters $\bar s_1,\dots,\bar s_h$ are chosen so that the set $\{\pm s_1,\dots,\pm s_m,\pm \bar s_1,\dots,\pm \bar s_h\}$ exhausts  the coprime classes module $q$.
We thus obtain a new lens space $\overline L$ of dimension $2h-1 = \phi(q)-2m-1$.

By \cite[Thm.~3.6]{Ik88}, if $q$ is prime, $L$ and $L'$ in $\mathfrak L_0(q,m)$ are $p$-isospectral for all $p$ if and only if $\overline L$ and $\overline{ L'}$ are $p$-isospectral for all $p$.

Now, for each  $q=r^2$, $r$ an odd prime, $r\not\equiv0\pmod3$, in Theorem~\ref{thm4:isolattices} we have obtained lens spaces $L$ and $L'$ in $\mathfrak L_0(r^2,3)$ that are $p$-isospectral for all $p$. Now, by an extension of Ikeda's argument  in \cite[Thm 3.6]{Ik88} for $q$  prime ---to be sketched below---  one can show that the associated lens spaces $\overline L$ and $\overline {L'}$ are $p$-isospectral for all $p$. These lens spaces have dimension $2h-1=\phi(r^2)-7 = r^2-r-7$, a quantity that tends to infinity when $r$ does, thus the assertion in the theorem immediately follows.

We now explain why Ikeda's argument also works  in the case $q=r^2$, $r$ prime. One has that $L,L'$ are isospectral for every $p$ if and only if they have the same generating functions (see \cite[Thm 2.5]{Ik88}). Thus, one needs to show that the analogous sums  for  $\overline L$ and $\overline{L'}$ are equal to each other.

The generating function for $L$ is given as a sum over the elements in the cyclic group generated by $g$ (see \cite[Thm 2.5]{Ik88}), which can be split  into a subsum over $g^k$ with $(k,q)=1$ plus a subsum over $g^k$ with $(k,q)=r$ plus a term corresponding to the identity element (i.e.\ $k=0$) and similarly for the generating function for $L'$, with $g'$ in place of $g$.   As asserted, the total sums are equal to each other for $L$ and $L'$.

It turns out that to prove the assertion for $\overline L$ and $\overline{L'}$ it suffices to show that the subsums just mentioned are equal to each other for $L$ and $L'$ (the contribution for $k=0$ is the same in both cases). But it is not hard to show that this is true for the second subsums (hence also for the first ones) for the lens spaces corresponding to the lattices in Theorem~\ref{thm4:isolattices}, by taking into account  that both lattices  are of the form $\mathcal L(q;s_1,\ldots, s_n)$ with $s_i\equiv\pm 1 \pmod q$. This concludes the proof.
\end{proof}

\begin{rem}
In Section~\ref{sec:finiteness} we have seen that the finite set of $N_{\mathcal L}^{\mathrm{red}}(k,\zz )$  determines whether two $q$-congruence lattices are $\norma{\cdot}^*$-isospectral.
Moreover, we point out that these numbers also determine explicitly each individual $p$-spectrum of a lens space $L=\Gamma\ba S^{2m-1}$ for $0\le p \le n=2m-1$.
Indeed, by Proposition~\ref{prop2:p-spectrum}, the multiplicities in the $p$-spectrum of $L$ depend only on the numbers $\dim V_{\pi_{k,p}}^\Gamma$ and $\dim V_{\pi_{k,p+1}}^\Gamma$ which, by expression \eqref{eq3:dim V_k,p^Gamma}, are determined by the $N_{\mathcal L}(k,\zz )$ which, in turn, can be computed by using equation \eqref{eq4:N^*-N^red} if we know the numbers $N_{\mathcal L}^{\mathrm{red}}(k,\zz )$.
\end{rem}

\begin{rem}\label{rem7:orbifolds}
If the discrete subgroup $\Gamma$ of $\SO(n+1)$ acts possibly with fixed points on $S^n$, then $\Gamma\ba S^n$ is a good orbifold.
For instance, in our case, if we take $L(q;s_1,\dots,s_m)$ as in \eqref{eq3:L(q;s)} with $s_1,\dots,s_m$ not necessarily coprime to $q$ and $\gcd(q,s_1,\dots,s_m)=1$, we obtain an \emph{orbifold lens space}.
See \cite{Shams} for an extension of Ikeda's result to orbifold lens spaces.

Most of the results in this paper  also work  for orbifold lens spaces.
For instance, the determination of the $p$-spectrum in Theorem~\ref{thm3:dim V_k,p^Gamma} via Proposition~\ref{prop2:p-spectrum} and the characterizations in Theorem~\ref{thm3:characterization} between lens spaces and congruence lattices.
Furthermore, Section~\ref{sec:finiteness} also works  for congruence lattices $\mathcal L(q;s_1, \dots,s_m)$ without the assumption that the $s_j$ are coprime to $q$.
Proposition~\ref{prop7:lens-non-strongly} is also valid in this context; that is, strongly isospectral orbifold lens spaces are necessarily isometric.
\end{rem}

We now show that the lens spaces constructed in Section~\ref{sec:families} are homotopically equivalent to each other.
We note that they cannot be simply homotopically equivalent (see \cite[\S31]{Co}) since in this case they would be homeomorphic.

\begin{lemma} \label{lem:homotequiv}
The lens spaces $L(r^2t;1, 1+rt, 1+3rt)$, $L(r^2t;1, 1-rt, 1-3rt)$, $r\not\equiv 0\pmod 3$, associated to the congruence lattices in Theorem~\ref{thm4:isolattices} are homotopically equivalent to each other.
\end{lemma}

\begin{proof}
We have seen that $L=L(q; \theta^0, \theta^1, \theta^{3})$ and $L'=L(q; \theta^0, \theta^{-1}, \theta^{-3})$, where $\theta =1+rt$.
The condition for homotopy equivalence of $L$ and $L'$ (see \cite[(29.6)]{Co}) is that $\pm \theta^{8}\equiv d^3 \pmod {r^2t}$, for some $d \in\Z$.

We claim that $r$ divides $\phi(r^2t)$. Indeed, we can write $q= \prod_j p_j^{2v_{p_j}(r)+ v_{p_j}(t)}$ a product over primes $p_j$. We have
$$ \phi(q)= r \prod_j p_j^{v_{p_j}(r)+ v_{p_j}(t)-1}(p_j-1).$$
Furthermore, this implies that  if $r$ is odd then $2r$ divides $\phi(r^2t)$.

We first assume that $rt$ is odd. Then $H:=\Z_{r^2t}^\times$ is a cyclic group of order $\phi(r^2t)$.
Thus, if $\omega$ is a generator of this group, then $\omega ^{\phi(r^2t)/2r}$ has order $2r$. Hence, since $H$ is cyclic,  $\theta = \omega^{\pm h}$ for some  $h=j\phi(r^2t)/2r$ with $(j,2r)=1$.
Now, if $3$ divides $\phi(r^2t)$, since $(3,2r)=1$, then   $\theta = \left(\omega^\frac{\phi(r^2t)j}{2r3}\right)^3$, as asserted.  If $3$ does not divide $\phi(r^2t)$, then the map $x\mapsto x^3$ is surjective, hence $\theta$ is again in the image.
This proves the assertion for $rt$ odd.

In case  $rt$ is even, then $\Z_{r^2t}^\times$ is a cyclic group $H$ times an abelian 2-group $K$.
Again the map $x\mapsto x^3$ in $K$ is surjective.
By a similar argument as before we show that $\theta$ is in the image of $x\mapsto x^3$ in $H$.
\end{proof}

\begin{rem}
In \cite{DR}, P.~Doyle and the third named author showed examples of disconnected flat orbifolds in dimension two that are $p$-isospectral for every $p$ but are not strongly isospectral.
\end{rem}

\section{$\tau$-isospectrality}\label{sec:tau-isospectrality}
In this section we give complementary spectral information, showing in a direct way the non-strong isospectrality of the pairs in Theorem~\ref{thm7:all-p-iso-lens}.
To make the computations easier, we consider in the sequel only the simplest pair
\begin{align*}
L&=L(49;1,6,15),\\
L'&=L(49;1,6,20)
\end{align*}
of non-isometric lens spaces $p$-isospectral for all $p$.
This pair is associated with the first row in Table~\ref{table:m=3} and it is isometric to the first pair in the family in Theorem~\ref{thm7:all-p-iso-lens} (see \eqref{eq6:ex-m=3-q=49}).
We denote by $\Gamma$ and $\Gamma'$ the finite cyclic subgroups of the torus $T\subset SO(6)$ of order $q=49$ such that $L=\Gamma\ba S^{5}$ and $L'=\Gamma'\ba S^{5}$.

We write $G=\SO(2m)$ and $K=\SO(2m-1)$ as in Section~\ref{sec:prelim}.
Any representation $\tau$ of $K$ induces a strongly elliptic natural operator $\Delta_{\tau,\Gamma}$ on the smooth sections of a natural bundle on a spherical space form $\Gamma\ba S^{2m-1}$.
By using representation theory, we will exhibit many choices of representations $\tau$ of $K$ such that $L$ and $L'$ are not $\tau$-isospectral.

\begin{lemma}\label{prop7:no_tau-iso}
The lens spaces $L=L(49;1,6,15)$ and $L'=L(49;1,6,20)$ are not $\tau$-isospectral for every irreducible representation $\tau$ of $\SO(5)$ with highest weight of the form $b_1\varepsilon_1 + b_2\varepsilon_2$ where
\begin{equation}\label{eq7:tau-weights}
4\geq b_1\geq 3\geq b_2\geq0.
\end{equation}
\end{lemma}
\begin{proof}
We choose $\Lambda_0=4\varepsilon_1+3\varepsilon_2$ and we let $\pi_{\Lambda_0}$ be the irreducible representation  of $G$ with highest weight $\Lambda_0$.
The Casimir element $C$ acts on $\pi_{\Lambda_0}$ by $\lambda_0=\lambda(C,\pi_{\Lambda_0} ) = \langle \Lambda_0+\rho,\Lambda_0+\rho\rangle-\langle \rho,\rho\rangle = (6^2+4^2)-(2^2+1^2)=47$.
By \eqref{eq2:mult_lambda}, the multiplicity $d_{\lambda_0}(\tau,\Gamma)$ of the eigenvalue $\lambda_0$ of $\Delta_{\tau,\Gamma}$ is
\begin{equation}\label{eq:dlambda}
d_{\lambda_0}(\tau,\Gamma)=\sum_\pi \dim V_\pi^{\Gamma} [\tau:\pi],
\end{equation}
where the sum is over the   irreducible representations $\pi$ of $\SO(6)$ such that $\lambda(C,\pi)=\lambda_0=47$.
A similar expression is valid for $d_{\lambda_0}(\tau,\Gamma')$.

Now if $\pi$ has highest weight $\Lambda=a_1\varepsilon_1+a_2\varepsilon_2+a_3\varepsilon_3$ ($a_i\in\Z$ and $a_1\geq a_2\geq|a_3|$) and  $\lambda(C,\pi)=\lambda_0$, then we have  $(a_1+2)^2+(a_2+1)^2+a_3^2=\langle\Lambda+\rho, \Lambda+\rho\rangle = \langle\Lambda_0+\rho, \Lambda_0+\rho\rangle =52.$
By taking congruence modulo $4$, we see that the numbers $a_1+2>a_2+1>a_3$ are even.
Hence
$
\left(\frac{a_1+2}{2}\right)^2 + \left(\frac{a_2+1}{2}\right)^2 + \left(\frac{a_3}{2}\right)^2=13.
$
It is easy to check that this implies that $a_1=4$, $a_2=3$ and $a_3=0$, therefore $\Lambda=\Lambda_0$ and hence, there is only one irreducible representation $\pi$ of $G$ with $\lambda(C,\pi)=47$, namely $\pi =\pi_{\Lambda_0}$.

On the other hand, by Lemma~\ref{lem3:L_Gamma} we have that $\dim V_{\pi_0}^{\Gamma} = \sum_{\mu\in \mathcal L} m_{\pi_0}(\mu)$, where $\mathcal L$ is the associated congruence lattice given by \eqref{eq3:Lambda(q;s)} and similarly for $\mathcal L'$.
We compute by using Sage~\cite{Sage} the weights of $\pi_0$ (i.e.\ $\mu\in\Z^m$ such that $m_{\pi_0}(\mu)>0$) and their respective multiplicities:
$$
\begin{array}{ll@{\quad\qquad}ll@{\qquad}ll}
4\varepsilon_1 + 3\varepsilon_2                    & 1 &
4\varepsilon_1 + 2\varepsilon_2 \pm  \varepsilon_3 & 1 &
3\varepsilon_1 + 3\varepsilon_2 \pm  \varepsilon_3 & 2  \\
3\varepsilon_1 + 2\varepsilon_2 \pm 2\varepsilon_3 & 2 &
4\varepsilon_1 +  \varepsilon_2                    & 1 &
3\varepsilon_1 + 2\varepsilon_2                    & 4  \\
3\varepsilon_1 +  \varepsilon_2 \pm  \varepsilon_3 & 4 &
2\varepsilon_1 + 2\varepsilon_2 \pm  \varepsilon_3 & 5 &
3\varepsilon_1                                     & 4  \\
2\varepsilon_1 +  \varepsilon_2                    & 9 &
 \varepsilon_1 +  \varepsilon_2 \pm \varepsilon_3  &12 &
 \varepsilon_1                                     &16
\end{array}
$$
(only the dominant weights are shown, since weights in the same Weyl group orbit have the same multiplicity).

A weight $\mu=\sum_i a_i\varepsilon_i$ is in $\mathcal L$ (resp.\ $\mathcal L'$) if and only if
$$
a_1+6a_2+15a_3\equiv0\pmod{49}\quad (\text{resp.\ } a_1+6a_2+20a_3\equiv0\pmod{49}).
$$

With computer aid one checks that the weights of $\pi_{\Lambda_0}$ that satisfy these congruences (i.e lying in $\mathcal L$ (resp.\ $\mathcal L'$) are $\pm(4\varepsilon_1+3\varepsilon_3)$, $\pm(-\varepsilon_1+2\varepsilon_2-4\varepsilon_3)$ and $\pm(3\varepsilon_1-3\varepsilon_2+\varepsilon_3)$ (resp.\ $\pm(-3\varepsilon_2-4\varepsilon_3)$ and $\pm(-3\varepsilon_1+2\varepsilon_2+2\varepsilon_3)$).
Taking into account their multiplicities we obtain that $\dim V_{\pi_{\Lambda_0}}^{\Gamma} = 2+2+4=8$ and $\dim V_{\pi_{\Lambda_0}}^{\Gamma'} = 2+4=6$.

By applying  the branching law from $\SO(6)$ to $\SO(5)$ (see for instance \cite[Thm.~8.1.4]{GW}) , for every irreducible representation $\tau$ of $\SO(5)$ with highest weight of the form $b_1\varepsilon_1 + b_2\varepsilon_2$ with $ 4\geq b_1\geq 3\geq b_2\geq 0$   one has that $[\tau,{\pi_{\Lambda_0}}_{|K}] =1$.

Thus, we obtain from \eqref{eq:dlambda} that $$d_{\lambda_0}(\tau,\Gamma)=\dim V_{\pi_0}^{\Gamma}[\tau:\pi_0]=8, \qquad d_{\lambda_0}(\tau,\Gamma')=6. $$
Thus, $L$ and $L'$ cannot be $\tau$-isospectral for any $\tau$ as in the statement.
\end{proof}

\begin{rem}\label{rem7:strongly_tau-iso}
We note that the assertion in Lemma~\ref{prop7:no_tau-iso}  followed by comparing the multiplicities of $\lambda=\lambda(C,\pi)$ for only one choice of $\pi\in\widehat G$ satisfying $\dim V_{\pi}^\Gamma \neq \dim V_{\pi}^{\Gamma'}$.
By computer methods using Sage~\cite{Sage}, we have checked that there are many different  choices of $\pi$ that allow to find  many other $K$-types $\tau$ such that the lens spaces $L$ and $L'$ are not $\tau$-isospectral.
\end{rem}

We believe that there are only finitely many irreducible representations $\tau$ of $\SO(5)$ such that $L$ and $L'$ are $\tau$-isospectral.

\bigskip

\noindent\emph{Acknowledgement.}
The authors wish to thank Peter G.~Doyle for very stimulating conversations on the subject of this paper.
The first named author wishes to thank the support of the Oberwolfach Leibniz Fellows programme (Germany) in May--July 2013 and in August--November 2014.

\bibliographystyle{plain}

\end{document}